\documentclass[10pt]{article}

\usepackage{amssymb,amsthm,amsmath,amsfonts}
\usepackage{graphicx}
\usepackage[usenames]{color}

\newtheorem{thm}{Theorem}[section]

\newtheorem{lem}{Lemma}[section]

\newtheorem{rmk}{Remark}[section]
\newtheorem{exam}{Example}[section]
\title{A New Approach to Nonlinear Constrained Tikhonov Regularization}
\date{\today}
\author{Kazufumi Ito\footnote{Department of Mathematics and Center for
Research in Scientific Computation, North Carolina State University,
Raleigh, North Carolina 27695, USA (kito@unity.ncsu.edu). } \and Bangti
Jin\footnote{Department of Mathematics and Institute for Applied Mathematics and Computational Science, Texas
A\&M University, College Station, Texas 77843-3368, USA (btjin@math.tamu.edu).}}
\begin{document}
\maketitle

\begin{abstract}
We present a novel approach to nonlinear constrained Tikhonov regularization from the
viewpoint of optimization theory. A second-order sufficient optimality condition is
suggested as a nonlinearity condition to handle the nonlinearity of the forward operator.
The approach is exploited to derive convergence rates results for a priori as well as a
posteriori choice rules, e.g., discrepancy principle and balancing principle, for
selecting the regularization parameter. The idea is further illustrated on a general
class of parameter identification problems, for which (new) source and nonlinearity
conditions are derived and the structural property of the nonlinearity term is revealed.
A number of examples including identifying distributed parameters in elliptic
differential equations are presented.

\end{abstract}

\section{Introduction}
In this paper, we discuss a robust method for solving ill-posed nonlinear operator
equations
\begin{equation}\label{eqn:nonlineq}
K(u)=g^\delta,
\end{equation}
where $g^\delta\in H$ denotes the noisy data, with its accuracy relative to the exact
data $g^\dagger=K(u^\dagger)$ ($u^\dagger\in X$ is the exact solution) measured by the
noise level $\delta=\|g^\dagger-g^\delta\|$. Here the nonlinear operator $K:X\rightarrow
H$ is Fr\'{e}chet differentiable, and the spaces $X$ and $H$ are Hilbert spaces.

In practice, the unknown coefficient $u$ may be subjected to pointwise constraint, e.g.,
$u\geq c$ almost everywhere. This is especially true for distributed coefficient
estimation in differential equations to ensure the well-definedness of the operator $K$,
see, e.g., \cite{BanksKunisch:1989,Isakov:2006,Chavent:2009} for relevant examples. We
denote the constraint set by $\mathcal{C}\subset X$, and assume that it is closed and
convex and $u^\dagger \in\mathcal{C}$.

To obtain an accurate yet stable approximation, we employ the now classical approach of
minimizing the following Tikhonov functional
\begin{equation*}
J_\eta(u)=\tfrac{1}{2}\|K(u)-g^\delta\|^2+ \tfrac{\eta}{2}\,\|u\|^2,
\end{equation*}
where the two terms are the fidelity incorporating the information in the data $g^\delta$
and a regularization for stabilizing the problem, respectively. A faithful choice of the
fidelity depends on the statistics of noises corrupting the data
\cite{ClasonJinKunisch:2010}. The penalty is chosen to reflect a priori knowledge (such
as smoothness) and constraint on the expected solutions, and nonsmooth penalties may also
be adopted.

Consequently, we arrive at the following constrained optimization problem:
\begin{equation}\label{prob:constr}
\min_{u\in\mathcal{C}}\left\{J_\eta(u)\equiv\tfrac{1}{2}\|K(u)-g^\delta\|^2+\tfrac{\eta}{2}\,\|u\|^2\right\},
\end{equation}
given a regularization parameter $\eta>0$. The minimizer of \eqref{prob:constr} is
denoted by $u_\eta^\delta$, and respectively, the minimizer for the exact data
$g^\dagger$ by $u_\eta$. In practice, it is important to develop rules for determining
the scalar parameter $\eta>0$ automatically so as to obtain robust yet accurate
approximations $u^\delta_\eta$ to the exact solution $u^\dagger$. We will analyze the
discrepancy principle \cite{Morozov:1966}, which uses a precise knowledge of the noise
level $\delta$ and determines $\eta$ by $\|K(u^\delta_\eta)-g^\delta\|\sim \delta$, and
two heuristic rules (balancing principle and Hanke-Raus rule), which do not require a
knowledge of the noise level and are purely data-driven.

Since the pioneering works \cite{EnglKunischNeubauer:1989,SeidmanVogel:1989}, nonlinear
Tikhonov regularization (in the absence of constraints) has been intensively studied
\cite[Chap. 10]{EnglHankNeubauer:1996}. Various existence, stability and consistency
results were established, and diverse practical applications have been successfully
demonstrated. Also convergence rates results were derived for several choice rules, e.g.,
discrepancy principle and monotone error rule
\cite{ChaventKunisch:1994,TautenhahnJin:2003}. The essential ingredients of convergence
rates analysis are the source and nonlinearity conditions.

To derive a convergence rates result, extra conditions on the exact solution $u^\dagger$
are necessary \cite{EnglHankNeubauer:1996}. They are collectively known as source
conditions, and often impose certain smoothness assumptions on $u^\dagger$. In the
absence of constraints, it is usually expressed via range inclusion, e.g.,
$u^\dagger=K'(u^\dagger)^\ast w$ for some representer $w\in H$, or variational
inequalities. In this work, we shall use an appropriate source condition in the presence
of convex constraints, derived from the viewpoint of optimization theory as in
\cite{ChaventKunisch:1994}. We refer to \cite{Neubauer:1988,LorenzRosch:2010} for related
results on constrained linear inverse problems.

One also needs conditions on the operator $K$ to control its degree of nonlinearity. One
classical condition is that the derivative $K'(u)$ of the operator $K$ is Lipschitz
continuous with its Lipschitz constant $L$ satisfying $L\|w\|<1$
\cite{EnglKunischNeubauer:1989}. We shall propose a second-order sufficient condition on
the solution $u^\dagger$ as an alternative. It is much weaker than the classical one, yet
sufficient for analyzing the Tikhonov functional \eqref{prob:constr}, i.e., establishing
convergence rates results. The idea is further explored on a broad class of nonlinear
parameter identifications, by exploiting explicit structures of the adjoint operator
$K'(u^\dagger)^\ast$ and by revealing the structure of the crucial nonlinearity term
$\langle w, E(u,u^\dagger)\rangle$ for bilinear problems.

The rest of the paper is structured as follows. In Section \ref{sec:sosc}, we motivate
the source condition \eqref{scon} and the nonlinearity condition \eqref{ncon} using
optimization theory. In Section \ref{sec:err} we establish an a priori convergence rate
of the discrepancy principle \cite{Morozov:1966} and a posteriori convergence rates for
two heuristic rules (balancing principle \cite{ItoJinZou:2010,ItoJinTakeuchi:2010} and
Hanke-Raus rule \cite{HankeRaus:1996,EnglHankNeubauer:1996}). In Section \ref{sec:class}
we illustrate the approach on a general class of nonlinear parameter identification
problems. Here we shall derive new source and nonlinearity conditions, and reveal the
structural property of the nonlinearity term. Finally, the abstract theories are worked
out in Section \ref{sec:exam} for several concrete examples. An example is given, for
which the smallness assumption ($L\|w\|<1$) in the classical nonlinearity condition is
violated, whereas the proposed nonlinearity condition \eqref{ncon} always holds. Also,
detailed derivations are presented for three representative inverse coefficient problems.
Throughout the paper, we shall use the symbol $\langle\cdot,\cdot\rangle$ to denote both
inner products in Hilbert spaces and duality pairing, and the notation $c$ to denote a
generic constant which may change at different occurrences but does not depend on the
quantities of interest.

\section{Second-order sufficient condition}\label{sec:sosc}
In this section, we develop the new approach from the viewpoint of optimization theory.
We begin with a second-order necessary condition for the minimizer $u_\eta$. Then we
propose using a second-order sufficient condition as a nonlinearity condition, show its
connection with classical conditions, and establish its role in deriving basic error
estimates.

\subsection{Necessary optimality system}

Consider the following generic constrained Tikhonov regularization formulation
\begin{equation*}
   \min_{u\in\mathcal{C}}\phi(u,g^\dagger)+\eta\psi(u),
\end{equation*}
where the fidelity $\phi(u,g)$ is differentiable in the first argument, the penalty
$\psi(u)$ is convex and (weakly) lower semi-continuous, and the constraint set
$\mathcal{C}$ is convex and closed. Our derivation of the necessary optimality condition
follows from \cite{ItoKunisch:2008}.

Let $u_\eta$ be a minimizer of the problem, i.e.,
\begin{equation*}
\phi(u_\eta,g^\dagger)+\eta\psi(u_\eta)\leq
\phi(v,g^\dagger)+\eta\psi(v)\quad \forall v\in\mathcal{C}.
\end{equation*}
Since $\psi$ is convex, for $v=u_\eta+t\,(u-u_\eta) \in \mathcal{C}$ with $u\in
\mathcal{C}, 0<t\le1$
\begin{equation*}
\frac{\phi(v,g^\dagger)-\phi(u_\eta,g^\dagger)}{t}
\ge -\eta\,\frac{\psi(v)-\psi(u_\eta)}{t}\ge
-\eta\,(\psi(u)-\psi(u_\eta)).
\end{equation*}
By letting $t\to0^+$, we obtain the necessary optimality condition
\begin{equation} \label{opt}
\langle\phi'(u_\eta,g^\delta),u-u_\eta\rangle+
\eta(\psi(u)-\psi(u_\eta))\geq 0\quad \forall u\in\mathcal{C}.
\end{equation}
Now let  $\partial\psi(u)$ denote the subdifferential of the convex functional $\psi$ at
$u$, i.e.,
\begin{equation*}
     \partial\psi(u) = \{\xi\in X^\ast: \psi(\tilde u)\geq \psi(u)
     +\langle\xi,\tilde u-u\rangle\;\;\forall\tilde{u}\in X\}.
\end{equation*}
%It follows from inequality \eqref{opt} and the convexity of $\psi$ that in the absence of
%the constraint $\mathcal{C}$
%\begin{equation} \label{var}
%  -\frac{1}{\eta} \phi'(u_\eta,g^\dagger) \in \partial\psi(u_\eta).
% \end{equation}
Consequently, it follows from \eqref{opt} and the convexity of $\psi$ that if there
exists an element $\xi_\eta \in
\partial\psi(u_\eta)$ and let
\begin{equation*}
   \mu_\eta=\frac{1}{\eta}\left(\phi'(u_\eta,g^\dagger)+\eta\xi_\eta\right),
\end{equation*}
then we have
\begin{equation} \label{var1}
\left\{\begin{array}{l}
\phi'(u_\eta,g^\dagger) + \eta \xi_\eta-\eta\mu_\eta=0,\\
\langle\mu_\eta,u-u_\eta\rangle\geq 0\ \ \forall u\in\mathcal{C},\\
\xi_\eta\in\partial\psi(u_\eta).
\end{array}\right.
\end{equation}
Thus, $\mu_\eta \in X^*$ serves as a Lagrange multiplier for the constraint
$\mathcal{C}$, cf. \cite[Thm. 3.2]{MaurerZowe:1979}. If $\psi^\prime(u_\eta) \in X^*$
exists, then $\xi_\eta =\psi^\prime(u_\eta)$ and thus
$\mu_\eta=\frac{1}{\eta}\phi^\prime(u_\eta,g^\dagger)+\psi^\prime(u_\eta)$. In a more
general constrained optimization, the existence of $\mu_\eta \in X^*$ is guaranteed by
the regular point condition \cite{MaurerZowe:1979,ItoKunisch:2008}. The inequality
\eqref{var1} is the first-order optimality condition. We refer to
\cite{MaurerZowe:1979,ItoKunisch:2008} for a general theory of second-order conditions
(see also Lemma 2.1).

\subsection{Source and nonlinearity conditions}

Henceforth we focus on problem \eqref{prob:constr}. We will propose a new nonlinearity
condition based on a second-order sufficient optimality condition. To this end, we first
introduce the second-order error $E(u,\tilde{u})$ of the operator $K$
\cite{ItoKunisch:2008} defined by
\begin{equation*}
E(u,\tilde{u})=K(u)-K(\tilde{u})-K'(\tilde{u})(u-\tilde{u}).
\end{equation*}
which quantitatively measures the degree of nonlinearity, or pointwise linearization
error, of the operator $K$, and will be used in deriving our nonlinearity condition. We
also recall the first-order necessary optimality condition for $u_\eta$ (cf.,
\eqref{var1})
\begin{equation}\label{eqn:1stopteta}
  \left\{\begin{array}{l} K'(u_\eta)^\ast(K(u_\eta)-g^\dagger)+\eta
    u_\eta-\eta\mu_\eta=0,\\
   \langle\mu_\eta,u-u_\eta\rangle\geq0 \,\,\,\forall u\in\cal{C}, % \quad \mu_\eta \geq 0,
  \end{array}\right.
\end{equation}
where $\mu_\eta$ is a Lagrange multiplier for the constraint $\mathcal{C}$. In view of
the differentiability of the penalty, the Lagrange multiplier $\mu_\eta$ is explicitly
given by $u_\eta+\frac{1}{\eta}K'(u_\eta)^* (K(u_\eta)-g^\dagger)$.

Now we can derive a second-order necessary optimality condition for problem
\eqref{prob:constr}.

\begin{lem}\label{lem:2ndneccond}
The necessary optimality condition of a minimizer $u_\eta$ to the Tikhonov functional
$J_\eta$ with the exact data $g^\dagger$ is given by: for any $u\in\mathcal{C}$
\begin{equation}\label{eqn:2ndneccond}
\tfrac{1}{2}\|K(u_\eta)-K(u)\|^2+\tfrac{\eta}{2}\|u_\eta-u\|^2+\langle
K(u_\eta)-g^\dagger,E(u,u_\eta)\rangle+\eta\langle\mu_\eta,u-u_\eta\rangle\geq0,
\end{equation}
where $\mu_\eta$ is a Lagrange multiplier associated with the constraint $\mathcal{C}$.
\end{lem}
\begin{proof}
By the minimizing property of $u_\eta$, we have that for any $u\in\mathcal{C}$
\begin{equation*}
\tfrac{1}{2}\|K(u_\eta)-g^\dagger\|^2+\tfrac{\eta}{2}\|u_\eta\|^2
\leq\tfrac{1}{2}\|K(u)-g^\delta\|^2+\tfrac{\eta}{2}\|u\|^2.
\end{equation*}
Straightforward computations show the following two elementary identities
\begin{equation*}
\begin{aligned}
  \tfrac{1}{2}\|u_\eta\|^2-\tfrac{1}{2}\|u\|^2&=-\tfrac{1}{2}\|u_\eta-u\|^2-\langle
  u_\eta,u-u_\eta\rangle,\\
  \tfrac{1}{2}\|K(u_\eta)-g^\dagger\|^2-\tfrac{1}{2}\|K(u)-g^\dagger\|^2&=
  -\tfrac{1}{2}\|K(u_\eta)-K(u)\|^2-\langle
  K(u_\eta)-g^\dagger,K(u)-K(u_\eta)\rangle.
\end{aligned}
\end{equation*}
Upon substituting these two identities, we arrive at
\begin{equation}\label{eqn:nc1}
-\tfrac{1}{2}\|K(u_\eta)-K(u)\|^2-\tfrac{\eta}{2}\|u_\eta-u\|^2-\eta\langle
u_\eta,u-u_\eta\rangle-\langle
K(u_\eta)-g^\dagger,K(u)-K(u_\eta)\rangle\leq 0.
\end{equation}
Now, the optimality condition for the minimizer $u_\eta$ (cf. \eqref{eqn:1stopteta}) is
given by
\begin{equation*}
K'(u_\eta)^\ast(K(u_\eta)-g^\dagger)+\eta u_\eta-\eta\mu_\eta=0,
\end{equation*}
where $\mu_\eta$ is a Lagrange multiplier for the constraint $\mathcal{C}$. Consequently,
\begin{equation*}
\eta\langle u_\eta,u-u_\eta\rangle+\langle
K(u_\eta)-g^\dagger,K'(u_\eta)(u-u_\eta)\rangle-
\eta\langle\mu_\eta,u-u_\eta\rangle=0.
\end{equation*}
Assisted with this identity and the second-order error $E(u,u_\eta)$, inequality
\eqref{eqn:nc1} yields immediately the desired assertion.
\end{proof}

One salient feature of the optimality condition \eqref{eqn:2ndneccond} is that it is true
for any $u\in \mathcal{C}$ and thus it is a global one. Also, the term
$\langle\mu_\eta,u-u_\eta\rangle$ is always nonnegative. The necessary condition
\eqref{eqn:2ndneccond} may be strengthened as follows: there exist some $c_s\in [0,1)$
and $\epsilon'>0$ such that
\begin{equation}\label{eqn:suffccond}
  \begin{aligned}
     \tfrac{1}{2}\|K(u_\eta)-K(u)\|^2+\tfrac{\eta}{2}\|u_\eta-u\|^2+\langle&
     K(u_\eta)-g^\dagger,E(u,u_\eta)\rangle+\eta\langle\mu_\eta,u-u_\eta\rangle\\
     &\geq\tfrac{c_s}{2}\|K(u_\eta)-K(u)\|^2+\tfrac{\epsilon'\eta}{2}\|u-u_\eta\|^2 \quad\forall u\in\mathcal{C}.
  \end{aligned}
\end{equation}
That is, the left hand side of \eqref{eqn:suffccond} is coercive in the sense that it is
bounded below by the positive term $\tfrac{c_s}{2}\|K(u_\eta)-K(u)\|^2
+\tfrac{\epsilon'\eta}{2}\|u-u_\eta\|^2$. This condition is analogous to, but not
identical with, the positive definiteness requirement on the Hessian in classical
second-order conditions in optimization theory \cite{MaurerZowe:1979,ItoKunisch:2008}.
Nonetheless, we shall call condition \eqref{eqn:2ndneccond}/\eqref{eqn:suffccond} a
second-order necessary/sufficient optimality condition.

\begin{rmk}\label{rmk:nonsmooth}
The case of a general convex $\psi$ can be handled similarly using Bregman distance,
which is defined by $d_\xi(u,\tilde{u})=\psi(u)- \psi(\tilde{u})-\langle\xi,u-
\tilde{u}\rangle$ for any $\xi\in\partial\psi(\tilde{u})$. Then repeating the proof in
Lemma \ref{lem:2ndneccond} gives the following necessary optimality condition
$(\xi_\eta\in\partial \psi(u_\eta))$
\begin{equation*}
\tfrac{1}{2}\|K(u_\eta)-K(u)\|^2+\eta d_{\xi_\eta}(u,u_\eta)+\langle
K(u_\eta)-g^\dagger,E(u,u_\eta)\rangle+\eta\langle\mu_\eta,u-u_\eta\rangle\geq0.
\end{equation*}
All subsequent developments can be adapted to general penalty $\psi$ by replacing
$\tfrac{1}{2}\|u-\tilde{u}\|^2$ with $d_\xi(u,\tilde{u})$. We refer interested readers to
\cite{BoneskyKazimierskiMaass:2008} and references therein for properties of Bregman
distance.
\end{rmk}

Note that there always holds $u_\eta\rightarrow u^\dagger$ subsequentially as
$\eta\rightarrow0^\dagger$ \cite{EnglKunischNeubauer:1989}. Assume that
$\tfrac{g^\dagger-K(u_\eta)}{\eta}\rightarrow w$ weakly and $\mu_\eta\rightarrow
\mu^\dagger$ weakly in suitable spaces as $\eta\rightarrow0^+$. Then by taking limit in
equation \eqref{eqn:1stopteta} as $\eta\rightarrow0^+$, we arrive at the following source
condition.\\\\
\textbf{Condition} There exists a $w\in H$ and $\mu^\dagger\in X^*$ such that the exact
solution $u^\dagger$ satisfies
\begin{equation} \label{scon}
\left\{\begin{array}{c}-K'(u^\dagger)^\ast
w+u^\dagger-\mu^\dagger=0,\\
\langle\mu^\dagger,u-u^\dagger\rangle\geq 0\,\quad \forall u\in\mathcal{C}.
\end{array}\right.
\end{equation}

The source condition \eqref{scon} is equivalent to assuming the existence of a Lagrange
multiplier $w$ (for the equality constraint $K(u)=g^\dagger$) for the minimum-norm
problem
\begin{equation*}
  \min\quad \|u\| \quad\mbox{subject to  } K(u)=g^\dagger \mbox{
  and } u\in \mathcal{C},
\end{equation*}
and, hence, the source condition \eqref{scon} represents a necessary optimality condition
for the minimum-norm solution $u^\dagger$. We note that, in case of a linear operator
$K$, there necessarily holds the relation: $u^\dagger\in\overline{\mathrm{R}(K^\ast)}$,
the closure of the range space $\mathrm{R}(K^\ast)$. The source condition is stronger
since for ill-posed problems generally $ \overline{\mathrm{R}(K^\ast)}\neq
\mathrm{R}(K^\ast)$.

Now we can introduce our nonlinearity condition based on a second-order sufficient
condition.\\\\
\textbf{Condition} There exists some $\epsilon>0$ and $c_r\ge 0$ such that the exact
solution $u^\dagger$ satisfies
\begin{equation} \label{ncon}
\tfrac{c_r}{2}\|K(u)-K(u^\dagger)\|^2+
\tfrac{1}{2}\|u-u^\dagger\|^2-\langle w, E(u,u^\dagger)\rangle +
\langle \mu^\dagger,u-u^\dagger\rangle
\geq\tfrac{\epsilon}{2}\|u-u^\dagger\|^2\quad \forall
u\in\mathcal{C}.
\end{equation}

Here the elements $w$ and $\mu^\dagger$ are from the source condition \eqref{scon}. The
nonlinearity term $\langle w,E(u,u^\dagger)\rangle$ is motivated by the following
observation. If the source representer  $w$ does satisfy
$w=\lim_{\eta\rightarrow0}\frac{g^\dagger-K(u_\eta)}{\eta}$ (weakly), then
asymptotically, we may replace $K(u_\eta)-g^\dagger$ in \eqref{eqn:suffccond} with $-\eta
w$, divide \eqref{eqn:suffccond} by $\eta$ and take $\eta\to 0$ to obtain \eqref{ncon},
upon assuming the convergence of $\frac{1-c_s}{\eta}$ to a finite constant $c_r$. We
would like to point out that the constant $c_r$ may be made very large to accommodate the
nonlinearity of the operator $K$. The only possibly indefinite term is $\langle w,
E(u,u^\dagger)\rangle$. Hence, the analysis of $\langle w, E(u,u^\dagger)\rangle$ is key
to demonstrating the nonlinearity condition \eqref{ncon} for concrete operator equations.

\begin{rmk} On the nonlinearity condition \eqref{ncon}, we have the following two
remarks.
\begin{itemize}
\item[$(1)$] In case of constrained Tikhonov regularization, we may have $w=0$, which
    results in $\langle w,E(u,u^\dagger)\rangle=0$ and thus the nonlinearity
    condition \eqref{ncon} automatically holds. For example, if $\mathcal{C}=\{u:
    u\ge c\}$, with $c$ being a positive constant, and $u^\dagger=c$ is the exact
    solution $({i.e.},\, g^\dagger=K(u^\dagger))$, then $w=0$ and
    $\mu^\dagger=u^\dagger$ satisfy the source condition \eqref{scon}. Moreover, if
    the set $\{\mu^\dagger\neq0\}$ has a positive measure, then the term $\langle
    \mu^\dagger,u-u^\dagger\rangle$ provides a strictly positive contribution to
    \eqref{ncon}. These are possible beneficial consequences due to the presence of
    constraints.
\item[$(2)$] A classical nonlinearity condition \cite{EnglKunischNeubauer:1989} reads
    \begin{equation}\label{cond:clsnon}
       K'(u) \mbox{ is Lipschitz continuous
       with a Lipschitz constant } L \mbox{ satisfying } L\|w\|<1.
    \end{equation}
    There are several other nonlinearity conditions. A very similar condition
    \cite{HeinHofmann:2009} is given by
    $\|E(u,\tilde{u})\|\leq\frac{L}{2}\|u-\tilde{u}\|^2$ with $L\|w\|<1$, which
    clearly implies condition \eqref{ncon}. Another popular condition \cite[pp. 6,
    eq. (2.7)]{KaltenbacherNeubauerScherzer:2008} reads
    \begin{equation}\label{cond:clsnon2}
        \|E(u,\tilde{u})\|\leq c_E\|K(u)-K(\tilde{u})\|\|u-\tilde{u}\|.
    \end{equation}
    It has been used for analyzing iterative regularization methods. Clearly, it
    implies \eqref{ncon} for $c_r>(c_E\|w\|)^2$. We note that it implies
    \eqref{eqn:suffccond} after applying Young's inequality.
\end{itemize}
\end{rmk}

The following lemma shows that the proposed nonlinearity condition \eqref{ncon} is much
weaker than the classical one, cf. \eqref{cond:clsnon}. Similarly one can show this for
condition \eqref{cond:clsnon2}. Therefore, the proposed approach does cover the classical
results.
\begin{lem}
Condition \eqref{cond:clsnon} implies condition \eqref{ncon}.
\end{lem}
\begin{proof}
A direct estimate shows that under condition \eqref{cond:clsnon}, we have
\begin{equation*}
\begin{aligned}
|\langle w,E(u,u^\dagger)\rangle|\leq \|w\|\|E(u,u^\dagger)\|\leq \|w\|\cdot\tfrac{L}{2}\|u-u^\dagger\|^2
\end{aligned}
\end{equation*}
by $\|E(u,u^\dagger)\|\leq\tfrac{L}{2}\|u-u^\dagger\|^2$ from the Lipschitz continuity of
the operator $K'(u)$. Consequently,
\begin{equation*}
\begin{aligned}
   \tfrac{1}{2}\|u-u^\dagger\|^2-\langle w, E(u,u^\dagger)\rangle+\langle\mu^\dagger,u-u^\dagger\rangle&\geq
   \tfrac{1}{2}\|u-u^\dagger\|^2-\tfrac{1}{2}L\|w\|\|u-u^\dagger\|^2+\langle\mu^\dagger,u-u^\dagger\rangle\\
   &\geq \tfrac{1-L\|w\|}{2}\|u-u^\dagger\|^2,
\end{aligned}
\end{equation*}
by noting the relation $\langle \mu^\dagger,u-u^\dagger\rangle\geq0$ for any
$u\in\mathcal{C}$. This shows that condition \eqref{ncon} holds with
$\epsilon=1-L\|w\|>0$ and $c_r=0$.
\end{proof}

\begin{rmk}
The condition $L\|w\|<1$ is used for bounding the nonlinearity term $\langle
w,E(u_\eta,u^\dagger)\rangle$ from above. This is achieved by Cauchy-Schwarz inequality,
and thus the estimate might be too pessimistic since in general $\langle
w,E(u,u^\dagger)\rangle$ can be either indefinite or negative. This might explain the
effectiveness of Tikhonov regularization in practice even though assumption
\eqref{cond:clsnon} on the solution $u^\dagger$ and the operator $K(u)$ may be not
verified.
\end{rmk}

\subsection{Basic error estimates}

We derive two basic error estimates under the source condition \eqref{scon} and the
nonlinearity condition \eqref{ncon}: the approximation error $\|u_\eta-u^\dagger\|$ due
to the use of regularization and the propagation error $\|u_\eta^\delta-u_\eta\|$ due to
the presence of data noises. These estimates are useful for analyzing convergence rates
of some parameter selection rules as $\eta\to0^+$
\cite{EnglHankNeubauer:1996,TautenhahnJin:2003}.
\begin{lem}\label{lem:err}
Assume that conditions \eqref{scon} and \eqref{ncon} hold. Then the approximation error
$\|u_\eta-u^\dagger\|$ satisfies
\begin{equation*}
\|u_\eta-u^\dagger\|
\leq\epsilon^{-\frac{1}{2}}\|w\|\frac{\sqrt{\eta}}{\sqrt{1-c_r\eta}}
\quad \mbox{and}\quad \|K(u_\eta)-g^\dagger\|\leq
\frac{2\eta}{1-c_r\eta}\|w\|.
\end{equation*}
Moreover, if there exists some $\epsilon'>0$ independent of $\eta$ such that the
second-order sufficient optimality condition \eqref{eqn:suffccond} holds for all
$u\in\mathcal{C}$, then the propagation error $\|u_\eta^\delta-u_\eta\|$ satisfies
\begin{equation*}
\|u_\eta^\delta-u_\eta\|\leq\frac{1}{\sqrt{\epsilon'c_s}}\frac{\delta}{\sqrt{\eta}}\quad\mbox{and}\quad
\|K(u_\eta^\delta)-K(u_\eta)\|\leq \frac{2\delta}{c_s}.
\end{equation*}
\end{lem}
\begin{proof}
The minimizing property of the approximation $u_\eta$ and the relation
$g^\dagger=K(u^\dagger)$ imply
\begin{equation*}
\tfrac{1}{2}\|K(u_\eta)-g^\dagger\|^2+\tfrac{\eta}{2}\|u_\eta\|^2\leq
\tfrac{\eta}{2}\|u^\dagger\|^2.
\end{equation*}
The source condition \eqref{scon} and Cauchy-Schwarz inequality give
\begin{equation*}
\begin{aligned}
\tfrac{1}{2}\|K(u_\eta)-g^\dagger\|^2+\tfrac{\eta}{2}\|u_\eta-u^\dagger\|^2&\leq
-\eta\langle u^\dagger,u_\eta-u^\dagger\rangle\\
&=-\eta\langle w, K'(u^\dagger)(u_\eta-u^\dagger)\rangle-\eta\langle\mu^\dagger, u_\eta-u^\dagger\rangle\\
&=-\eta\langle w,K(u_\eta)-g^\dagger\rangle+\eta\langle w,E(u_\eta,u^\dagger)\rangle-\eta\langle\mu^\dagger,u_\eta-u^\dagger\rangle\\
&\leq\eta\|w\|\|K(u_\eta)-g^\dagger\|+\eta\langle w,E(u_\eta,u^\dagger)\rangle-\eta\langle\mu^\dagger,u_\eta-u^\dagger\rangle.
\end{aligned}
\end{equation*}
By appealing to the nonlinearity condition \eqref{ncon}, we arrive at
\begin{equation*}
\tfrac{1-c_r\eta}{2}\|K(u_\eta)-g^\dagger\|^2+\tfrac{\epsilon\eta}{2}\|u_\eta-u^\dagger\|^2\leq
\eta\|w\|\|K(u_\eta)-g^\dagger\|.
\end{equation*}
Consequently, by ignoring the term $\frac{\epsilon\eta}{2}\|u_\eta-u^\dagger\|^2$, we
derive the estimate
\begin{equation*}
\|K(u_\eta)-g^\dagger\|\leq \tfrac{2\eta}{1-c_r\eta}\|w\|,
\end{equation*}
and meanwhile, by invoking Young's inequality, we have
\begin{equation*}
\tfrac{1-c_r\eta}{2}\|K(u_\eta)-g^\dagger\|^2+\tfrac{\epsilon\eta}{2}\|u_\eta-u^\dagger\|^2\leq
\tfrac{1}{2(1-c_r\eta)}\eta^2\|w\|^2+\tfrac{1-c_r\eta}{2}\|K(u_\eta)-g^\dagger\|^2,
\end{equation*}
i.e., $\|u_\eta-u^\dagger\|\leq\epsilon^{-\frac{1}{2}}\|w\|
\tfrac{\sqrt{\eta}}{\sqrt{1-c_r\eta}}$. This shows the first assertion.

Next we turn to the propagation error $\|u_\eta^\delta-u_\eta\|$. We use the optimality
of the minimizer $u_\eta^\delta$ to get
\begin{equation}\label{eqn:app}
\tfrac{1}{2}\|K(u_\eta^\delta)-g^\delta\|^2+\tfrac{\eta}{2}\|u_\eta^\delta-u_\eta\|^2\leq
\tfrac{1}{2}\|K(u_\eta)-g^\delta\|^2-\eta\langle
u_\eta,u_\eta^\delta-u_\eta\rangle.
\end{equation}
Upon substituting the optimality condition of $u_\eta$ (cf. \eqref{eqn:1stopteta}), i.e.,
\begin{equation*}
\eta u_\eta=-K'(u_\eta)^\ast (K(u_\eta)-g^\dagger)+\eta\mu_\eta,
\end{equation*}
into \eqref{eqn:app}, we arrive at
\begin{equation*}
\begin{aligned}
\tfrac{1}{2}\|K(u_\eta^\delta)-g^\delta\|^2&+\tfrac{\eta}{2}\|u_\eta^\delta-u_\eta\|^2\leq
\tfrac{1}{2}\|K(u_\eta)-g^\delta\|^2+\langle
K(u_\eta)-g^\dagger,K'(u_\eta)(u_\eta^\delta-u_\eta)\rangle-\eta\langle\mu_\eta,u_\eta^\delta-u_\eta\rangle \\
&=\tfrac{1}{2}\|K(u_\eta^\delta)-g^\delta\|^2+\tfrac{1}{2}\|K(u_\eta^\delta)-K(u_\eta)\|^2-\langle
K(u_\eta^\delta)-g^\delta,K(u_\eta^\delta)-K(u_\eta)\rangle\\
&\quad+\langle
K(u_\eta)-g^\dagger,K'(u_\eta)(u_\eta^\delta-u_\eta)\rangle-\eta\langle\mu_\eta,u_\eta^\delta-u_\eta\rangle.
\end{aligned}
\end{equation*}
Now the second-order error $E(u_\eta^\delta,u_\eta)$ and the Cauchy-Schwarz inequality
yield
\begin{equation*}
\begin{aligned}
&\tfrac{1}{2}\|K(u_\eta^\delta)-K(u_\eta)\|^2+\tfrac{\eta}{2}\|u_\eta^\delta-u_\eta\|^2+\eta\langle\mu_\eta,u_\eta^\delta-u_\eta\rangle\\
\leq&-\langle K(u_\eta)-g^\delta,K(u_\eta^\delta)-K(u_\eta)\rangle+\langle
K(u_\eta)-g^\dagger,K'(u_\eta)(u_\eta^\delta-u_\eta)\rangle\\
=&\langle g^\delta-g^\dagger,K(u_\eta^\delta)-K(u_\eta)\rangle-\langle
K(u_\eta)-g^\dagger, E(u_\eta^\delta,u_\eta)\rangle\\
\leq&\|g^\delta-g^\dagger\|\|K(u_\eta^\delta)-K(u_\eta)\|-\langle
K(u_\eta)-g^\dagger, E(u_\eta^\delta,u_\eta)\rangle.
\end{aligned}
\end{equation*}
Consequently, we have
\begin{equation*}
\tfrac{1}{2}\|K(u_\eta^\delta)-K(u_\eta)\|^2+\tfrac{\eta}{2}\|u_\eta^\delta-u_\eta\|^2+\langle K(u_\eta)-g^\dagger, E(u_\eta^\delta,u_\eta)\rangle
+\eta\langle\mu_\eta,u_\eta^\delta-u_\eta\rangle \leq\delta\|K(u_\eta^\delta)-K(u_\eta)\|.
\end{equation*}
Finally, the second-order sufficient optimality condition \eqref{eqn:suffccond} implies
\begin{equation*}
\|K(u_\eta^\delta)-K(u_\eta)\|\leq \tfrac{2\delta}{c_s}
\quad \mbox{and} \quad
\|u_\eta^\delta-u_\eta\|\leq \tfrac{1}{\sqrt{\epsilon'c_s}}\tfrac{\delta}{\sqrt{\eta}}.
\end{equation*}
This completes the proof of the lemma.
\end{proof}

\section{Applications to choice rules}\label{sec:err}
In this part, we illustrate the utility of the proposed nonlinearity condition
\eqref{ncon} for analyzing selection rules, i.e., a priori rule, discrepancy principle
\cite{Morozov:1966}, balancing principle \cite{JinZou:2009,ItoJinZou:2010,
ItoJinTakeuchi:2010} and Hanke-Raus rule \cite{HankeRaus:1996}, in deriving either a
priori or a posteriori error estimates. The a posteriori error estimates seem new for
nonlinear problems. Also a first consistency result is provided for the balancing
principle.

\subsection{A priori parameter choice \& discrepancy principle}

First, we give a convergence rates result for a priori parameter choices.
\begin{thm}\label{thm:apriori}
Under conditions \eqref{scon} and \eqref{ncon}, we have the following estimates
\begin{equation*}
 \begin{array}{l}
  \|u_\eta^\delta-u^\dagger\|\leq \epsilon^{-\frac{1}{2}}
  \left(\tfrac{\sqrt{1+2c_r\eta}}{\sqrt{\eta}}\delta
  +\tfrac{\sqrt{\eta}}{\sqrt{1-2c_r\eta}}\|w\|+\sqrt{2\|w\|\delta}\right),\\
  \|K(u_\eta^\delta)-g^\delta\|\leq\tfrac{1}{\sqrt{1-2c_r\eta}}\left(
  \sqrt{1+2c_r\eta}\delta+\tfrac{2}{\sqrt{1-2c_r\eta}}\eta\|w\|+\sqrt{2\|w\|\eta\delta}\right).
 \end{array}
\end{equation*}
\end{thm}
\begin{proof}
In view of the optimality of the minimizer $u_\eta^\delta$ and the source condition
\eqref{scon}, we have
\begin{equation*}
  \begin{aligned}
    \tfrac{1}{2}\|K(u_\eta^\delta)-g^\delta\|^2+\tfrac{\eta}{2}\|u_\eta^\delta-u^\dagger\|^2&\leq
    \tfrac{1}{2}\|K(u^\dagger)-g^\delta\|^2-\eta\langle u^\dagger,u_\eta^\delta-u^\dagger\rangle\\
    &=\tfrac{1}{2}\|K(u^\dagger)-g^\delta\|^2-\eta\langle
    w,K'(u^\dagger)(u_\eta^\delta-u^\dagger)\rangle-\eta\langle\mu^\dagger,u_\eta^\delta-u^\dagger\rangle.
  \end{aligned}
\end{equation*}
With the help of the second-order error $E(u,\tilde{u})$, we deduce
\begin{equation*}
  \begin{aligned}
    \tfrac{1}{2}\|K(u_\eta^\delta)-g^\delta\|^2+\tfrac{\eta}{2}\|u_\eta^\delta-u^\dagger\|^2-\eta\langle
     w, &E(u_\eta^\delta,u^\dagger)\rangle+\eta\langle\mu^\dagger,u_\eta^\delta-u^\dagger\rangle\\
      &\leq\tfrac{1}{2}\|K(u^\dagger)-g^\delta\|^2-\eta\langle w,K(u_\eta^\delta)-K(u^\dagger)\rangle.
  \end{aligned}
\end{equation*}
Now from the nonlinearity condition \eqref{ncon} and the Cauchy-Schwarz and Young's
inequalities, we obtain
\begin{equation*}
  \begin{aligned}
    \tfrac{1-2c_r\eta}{2}\|K(u_\eta^\delta)-g^\delta\|^2
    +\tfrac{\epsilon\eta}{2}\|u_\eta^\delta-u^\dagger\|^2&\leq
    \tfrac{1+2c_r\eta}{2}\|K(u^\dagger)-g^\delta\|^2-\eta\langle w,K(u_\eta^\delta)-K(u^\dagger)\rangle\\
    &\leq\tfrac{1+2c_r\eta}{2}\delta^2+\eta\|w\|\left(\|K(u_\eta^\delta)-g^\delta\|+\|g^\dagger-g^\delta\|\right)\\
    &\leq\tfrac{1+2c_r\eta}{2}\delta^2+\eta\|w\|\|K(u_\eta^\delta)-g^\delta\|+\eta\delta\|w\|,
  \end{aligned}
\end{equation*}
where we have made use of the inequality
\begin{equation*}
\|K(u_\eta^\delta)-K(u^\dagger)\|^2\leq 2\|K(u_\eta^\delta)-g^\delta\|^2 +
2\|K(u^\dagger)-g^\delta\|^2.
\end{equation*}
Using Young's inequality again and the fact that $c^2\leq a^2+b^2(a,b,c\geq0)$ implies
$c\leq a+b$ gives
\begin{equation*}
    \|u_\eta^\delta-u^\dagger\|\leq \epsilon^{-\frac{1}{2}}
    \left(\tfrac{\sqrt{1+2c_r\eta}}{\sqrt{\eta}}\delta+\tfrac{\sqrt{\eta}}{\sqrt{1-2c_r\eta}}\|w\|+\sqrt{2\delta\|w\|}\right).
\end{equation*}
Meanwhile, by ignoring the term $\tfrac{\epsilon}{2}\|u_\eta^\delta- u^\dagger\|^2$, we
deduce
\begin{equation*}
    \|K(u_\eta^\delta)-g^\delta\|\leq \tfrac{1}{\sqrt{1-2c_r\eta}}\left(
    \sqrt{1+2c_r\eta}\delta+\tfrac{2}{\sqrt{1-2c_r\eta}}\eta\|w\|+\sqrt{2\eta\delta\|w\|}\right).
\end{equation*}
This concludes the proof of the theorem.
\end{proof}

Therefore, the a priori choice $\eta\sim\delta$ achieves a convergence rate
$\mathcal{O}(\delta^{\frac{1}{2}})$ and $\mathcal{O}(\delta)$ for the error
$\|u_\eta^\delta-u^\dagger\|$ and for the residual $\|K(u_\eta^\delta)-g^\dagger\|$,
respectively, which coincide with that for the classical nonlinearity condition
\eqref{cond:clsnon} \cite{EnglKunischNeubauer:1989}.

Next we illustrate the proposed approach for the popular discrepancy principle due to
Morozov \cite{Morozov:1966}, i.e., we determine an optimal parameter $\eta$ by: for some
constant $c_m\geq1$
\begin{equation}\label{eqn:morozov}
 \|K(u_\eta^\delta)-g^\delta\|=c_m\delta.
\end{equation}
The principle is very useful if a reliable estimate of the noise level $\delta$ is
available. The rationale is that the accuracy of the solution $u_\eta^\delta$ should be
consistent with that of the data (in terms of residual). The consistency of the principle
is well-known, and also it achieves a convergence rate $\mathcal{O}(\delta^\frac{1}{2})$
under the classical nonlinearity condition \eqref{cond:clsnon}
\cite{EnglKunischNeubauer:1989}. The following result shows that the (weaker)
nonlinearity condition \eqref{ncon} can reproduce the canonical convergence rate
$\mathcal{O}(\delta^\frac{1}{2})$. We would like to remark that the principle can be
efficiently implemented by either the model function approach or quasi-Newton method
\cite{JinZou:2010}.
\begin{thm}\label{thm:morozov}
Let conditions \eqref{scon} and \eqref{ncon} be fulfilled, and $\eta^\ast$ be determined
by principle \eqref{eqn:morozov}. Then the solution $u_{\eta^\ast}^\delta$ satisfies the
following estimate
\begin{equation*}
    \|u_{\eta^\ast}^\delta-u^\dagger\|\leq\tfrac{1}{\sqrt{\epsilon}}\left(\sqrt{2(1+c_m)\|w\|}
    \delta^\frac{1}{2}+(1+c_m)\sqrt{c_r}\delta\right).
\end{equation*}
\end{thm}
\begin{proof}
The minimizing property of $u_{\eta^\ast}^\delta$ and the defining relation
\eqref{eqn:morozov} imply
\begin{equation*}
\|u_{\eta^\ast}^\delta\|^2\leq \|u^\dagger\|^2.
\end{equation*}
Upon utilizing the source condition \eqref{scon} and the second-order error
$E(u_{\eta^\ast}^\delta, u^\dagger)$, we deduce
\begin{equation*}
\begin{aligned}
  \tfrac{1}{2}\|u_{\eta^\ast}^\delta-u^\dagger\|^2&\leq-\langle u^\dagger,u_{\eta^\ast}^\delta-u^\dagger\rangle\\
      &=-\langle K'(u^\dagger)^*w+\mu^\dagger,u_{\eta^\ast}^\delta-u^\dagger\rangle\\
      &=-\langle w,K'(u^\dagger)(u_{\eta^\ast}^\delta-u^\dagger)\rangle-\langle\mu^\dagger,
          u_{\eta^\ast}^\delta-u^\dagger\rangle\\
      &=-\langle w,K(u_{\eta^\ast}^\delta)-K(u^\dagger)\rangle+\langle w,E(u^\delta_{\eta^\ast},
      u^\dagger)\rangle-\langle\mu^\dagger,u_{\eta^\ast}^\delta-u^\dagger\rangle.
\end{aligned}
\end{equation*}
Now the nonlinearity condition \eqref{ncon} yields
\begin{equation*}
\begin{aligned}
  \tfrac{\epsilon}{2}\|u_{\eta^\ast}^\delta-u^\dagger\|^2&\leq \|w\|\|K(u_{\eta^\ast}^\delta)-K(u^\dagger)\|
   +\tfrac{c_r}{2}\|K(u_{\eta^\ast}^\delta)-K(u^\dagger)\|^2 \\
    &\leq (c_m+1)\|w\|\delta+\tfrac{c_r}{2}(1+c_m)^2\delta^2,
\end{aligned}
\end{equation*}
where we have used the triangle inequality and \eqref{eqn:morozov} as follows
\begin{equation*}
\|K(u_{\eta^\ast}^\delta)-K(u^\dagger)\|\le
\|K(u_{\eta^\ast}^\delta)-g^\delta\|+\|K(u^\dagger)-g^\delta\|\le
(1+c_m)\delta.
\end{equation*}
The desired estimate follows immediately from these inequalities.
\end{proof}

\subsection{Two heuristic rules}
Next we apply the proposed approach to deriving a posteriori error estimates for two
heuristic selection rules, i.e., balancing principle
\cite{JinZou:2009,ItoJinTakeuchi:2010} and Hanke-Raus rule
\cite{HankeRaus:1996,EnglHankNeubauer:1996}. These rules were originally developed for
linear inverse problems and recently also for nonsmooth models
\cite{JinLorenz:2010,ItoJinTakeuchi:2010}, but their nonlinear counterparts have not been
studied. The subsequent derivations rely crucially on Lemma \ref{lem:err}, and thus
invoke the second-order sufficient optimality condition \eqref{eqn:suffccond}.

\subsubsection{Balancing principle}

There are several rules known under the name balancing principle \cite[Sect.
2.2]{ItoJinZou:2010}.  The variant under consideration is due to
\cite{ItoJinTakeuchi:2010}, which has been successfully applied to a variety of contexts,
including nonsmooth penalties \cite{ItoJinTakeuchi:2010}. It chooses an optimal
regularization parameter $\eta^\ast$ by minimizing
\begin{equation}\label{eqn:bal}
 \eta^\ast=\arg\min_{\eta\in[0,\|K\|^2]}\frac{F^{1+\gamma}(\eta)}{\eta},
\end{equation}
where $F(\eta)=J_\eta(u_\eta^\delta)$ is the value function, and $\gamma>0$ is a fixed
constant. We refer to \cite{ItoJinTakeuchi:2010} for fine properties of the function
$F(\eta)$. Generically, for the fidelity $\phi(u,g^\delta)$ and penalty $\psi(u)$, it is
equivalent to augmented Tikhonov functional $J(u,\lambda,\tau)$ recently derived from
Bayesian paradigm \cite{JinZou:2009}
\begin{equation*}
    J(u,\lambda,\tau)=\tau\phi(u,g^\delta)+ \lambda\psi(u)+
    \beta_0\lambda-\alpha_0\ln\lambda+\beta_1\tau-\alpha_1 \ln\tau,
\end{equation*}
which maximizes the posteriori probability density $p(u,\tau,\lambda|g^\delta)\propto
p(g^\delta|u,\tau)p(u,\tau,\lambda)$ with the scalars $\tau$ (noise precision) and
$\lambda$ (prior precision) both having Gamma distributions. Here the parameter pairs
$(\alpha_0,\beta_0)$ and $(\alpha_1,\beta_1)$ are closely related to the scale/shape
parameters in the Gamma distributions. The approach determines the regularization
parameter $\eta$ by $\eta=\lambda\tau^{-1}$, and in the case of $\beta_0=\beta_1=0$, the
selected parameter $\eta^\ast$ satisfies rule \eqref{eqn:bal} with the free parameter
$\gamma$ being fixed at the ratio $\frac{\alpha_1}{\alpha_0}$ \cite{ItoJinTakeuchi:2010}.

The name of the rule originates from the fact that the selected parameter $\eta^\ast$
automatically balances the penalty $\psi(u)=\|u\|^2$ with the fidelity
$\phi(u,g^\delta)=\|K(u)-g^\delta\|^2$ through the balancing relation
\begin{equation}\label{eqn:baleq}
  \gamma\eta^\ast\|u_{\eta^\ast}^\delta\|^2 =\|K(u_{\eta^\ast}^\delta)-g^\delta\|^2.
\end{equation}
This relation also shows clearly the role of the parameter $\gamma$ as a balancing
weight.

First, we give an a posteriori error estimate for the approximation
$u_{\eta^\ast}^\delta$.

\begin{thm}\label{thm:bal}
Let the conditions in Lemma \ref{lem:err} be fulfilled, $\eta^\ast$ be determined by rule
\eqref{eqn:bal}, and $\delta_*=\|K(u_{\eta^*}^\delta)-g^\delta\|$ be the realized
residual. Then the following estimate holds
\begin{equation*}
\|u_{\eta^\ast}^\delta-u^\dagger\|\leq
c\left(\epsilon^{-\frac{1}{2}}+\epsilon^{\prime-\frac{1}{2}}\tfrac{F^\frac{\gamma+1}{2}
(\delta)}{F^\frac{\gamma+1}{2}(\eta^\ast)}\right)\max(\delta,\delta_\ast)^\frac{1}{2}.
\end{equation*}
\end{thm}
\begin{proof}
By the triangle inequality, we have the error decomposition
\begin{equation*}
   \|u_\eta^\delta-u^\dagger\|\leq\|u_\eta^\delta-u_\eta\|+\|u_\eta-u^\dagger\|.
\end{equation*}
It suffices to bound the approximation error $\|u_\eta-u^\dagger\|$ and the propagation
error $\|u_\eta^\delta-u_\eta\|$. For the former, we have from the source condition
\eqref{scon} (cf. the proof of Lemma \ref{lem:err}) that
\begin{equation*}
  \begin{aligned}
   \tfrac{1}{2}\|u_{\eta^\ast}-u^\dagger\|^2&\leq-\langle w,K'(u^\dagger)(u_{\eta^\ast}-u^\dagger)\rangle-
   \langle\mu^\dagger,u_{\eta^\ast}-u^\dagger\rangle\\
   &=-\langle w, K(u_{\eta^\ast})-K(u^\dagger)\rangle + \langle w,E(u_{\eta^\ast},u^\dagger)\rangle
   -\langle\mu^\dagger,u_{\eta^\ast}-u^\dagger\rangle.
  \end{aligned}
\end{equation*}
Consequently, we get
\begin{equation*}
\begin{aligned}
   \tfrac{1}{2}\|u_{\eta^\ast}-u^\dagger\|^2-\langle w,E(u_{\eta^\ast},u^\dagger)\rangle+
   \langle\mu^\dagger,u_{\eta^\ast}-u^\dagger\rangle
   \leq\|w\|\|K(u_{\eta^\ast})-g^\dagger\|.
\end{aligned}
\end{equation*}
However, by the triangle inequality and  Lemma \ref{lem:err}, the term
$\|K(u_{\eta^\ast})-g^\dagger\|$ can be estimated by
\begin{equation*}
\begin{aligned}
  \|K(u_{\eta^\ast})-g^\dagger\| &\leq \|K(u_{\eta^\ast})-K(u_{\eta^\ast}^\delta)\|+\|K(u_{\eta^\ast}^\delta)-g^\delta\|+\|g^\delta-g^\dagger\|\\
    &\leq \tfrac{2\delta}{c_s}+\delta^\ast+\delta \leq\tfrac{2+2c_s}{c_s}\max(\delta,\delta_\ast).
\end{aligned}
\end{equation*}
This together with the nonlinearity condition \eqref{ncon} yields
\begin{equation*}
\begin{aligned}
  \tfrac{\epsilon}{2}\|u_{\eta^\ast}-u^\dagger\|^2\leq \|w\|\|K(u_{\eta^\ast})-g^\dagger\| + \tfrac{c_r}{2}\|K(u_{\eta^\ast})-g^\dagger\|^2,
\end{aligned}
\end{equation*}
i.e.,
\begin{equation*}
   \|u_{\eta^\ast}-u^\dagger\|\leq \tfrac{1}{\sqrt{\epsilon}}\left(2\tfrac{\sqrt{1+c_s}}{\sqrt{c_s}}\sqrt{\|w\|}\max(\delta,\delta_\ast)^\frac{1}{2}
    +\tfrac{2+2c_s}{c_s}\sqrt{c_r}\max(\delta,\delta_\ast)\right).
\end{equation*}
Next we estimate the propagation error $\|u_{\eta^\ast}^\delta-u_{\eta^\ast}\|$. The
minimizing property of the selected parameter $\eta^\ast$ implies
\begin{equation*}
\tfrac{1}{\eta^\ast}\leq\tfrac{F^{\gamma+1}(\delta)}{F^{\gamma+1}(\eta^\ast)}\tfrac{1}{\delta}.
\end{equation*}
By Lemma \ref{lem:err}, we have
\begin{equation*}
\begin{aligned}
\|u_{\eta^\ast}^\delta-u_{\eta^\ast}\|&\leq\tfrac{1}{\sqrt{\epsilon'}}\tfrac{1}{\sqrt{c_s\eta^\ast}}\delta
\leq\tfrac{1}{\sqrt{c_s\epsilon'}}\tfrac{F^\frac{\gamma+1}{2}(\delta)}{F^{\frac{\gamma+1}{2}}(\eta^\ast)}\delta^\frac{1}{2}\\
&\leq\tfrac{1}{\sqrt{c_s\epsilon'}}\tfrac{F^\frac{\gamma+1}{2}(\delta)}{F^{\frac{\gamma+1}{2}}(\eta^\ast)}\max(\delta,\delta_\ast)^\frac{1}{2}.
\end{aligned}
\end{equation*}
The desired estimates follows from the choice $c=\max(2\tfrac{\sqrt{1+c_s}}{\sqrt{c_s}
}\sqrt{\|w\|}+\tfrac{2+2c_s}{c_s}\sqrt{c_r}\max(\delta,\delta_\ast)^\frac{1}{2},\tfrac{1}{\sqrt{c_s}})$.
\end{proof}

We note that both $\delta$ and $\delta_*$ are naturally bounded, so the constant $c$ in
Theorem \ref{thm:bal} can be made independent of $\max(\delta,\delta_*)$. The estimate
provides an a posteriori check of the selected parameter $\eta^\ast$: if the realized
residual $\delta_*$ is far smaller than the expected noise level, then the prefactor
$F^{\frac{1+\gamma}{2}}(\delta)F^{-\frac{1+\gamma}{2}}(\eta^\ast)$ might blow up, and
hence, one should be very cautious about the reliability of the approximation
$x_{\eta^\ast}^\delta$. Despite a posteriori error estimates for the principle, its
consistency remains unaddressed, even for linear inverse problems. We make a first
attempt to this issue. First, we show a result on the realized residual $\delta_*$.
\begin{lem}\label{lem:phiconv}
Let the minimizer $\eta^\ast\equiv\eta^\ast(\delta)$ of rule \eqref{eqn:bal} be realized
in $(0,\|K\|^2)$. Then there holds
\begin{equation*}
   \|K(u_{\eta^\ast}^\delta)-g^\delta\|\rightarrow0\quad\mbox{ as } \delta\rightarrow 0.
\end{equation*}
\end{lem}
\begin{proof}
By virtue of \cite[Thm. 3.1]{ItoJinTakeuchi:2010}, the balancing equation
\eqref{eqn:baleq} is achieved at the local minimizer $\eta^\ast$. Consequently, it
follows from \eqref{eqn:baleq} and the optimality of the selected parameter $\eta^\ast$
that
\begin{equation}\label{eqn:valres}
\frac{[(1+\gamma^{-1})\|K(u_{\eta^\ast}^\delta)-g^\delta\|^2]^{\gamma+1}}{\eta^\ast}=\frac{F^{\gamma+1}(\eta^\ast)}{\eta^\ast}\leq
\frac{F^{\gamma+1}(\tilde{\eta})}{\tilde{\eta}},
\end{equation}
for any $\tilde\eta\in[0,\|K\|^2]$. However, with the choice $\tilde{\eta}=\delta$ and by
the optimality of the minimizer $u_{\tilde{\eta}}^\delta$, we have
\begin{equation*}
   \begin{aligned}
     F(\tilde{\eta})&\equiv\tfrac{1}{2}\|K(u_{\tilde{\eta}}^\delta)-g^\delta\|^2+\tfrac{\tilde\eta}{2}\|u_{\tilde\eta}^\delta\|^2\\
     &\leq \tfrac{1}{2}\|K(u^\dagger)-g^\delta\|^2 + \tfrac{\tilde{\eta}}{2}\|u^\dagger\|^2\\
     &\leq \tfrac{\delta^2}{2}+\delta\|u^\dagger\|^2\sim \delta.
   \end{aligned}
\end{equation*}
Hence, by noting the condition $\gamma>0$ and the a priori bound
$\eta^\ast\in[0,\|K\|^2]$, we deduce that the rightmost term in \eqref{eqn:valres} tends
to zero as $\delta\rightarrow0$. This shows the desired assertion.
\end{proof}

We can now state a consistency result.
\begin{thm}\label{thm:bpcon}
Let there exist some $M>0$ such that $\|u^\dagger\|\leq M$, and the assumption in Lemma
\ref{lem:phiconv} be fulfilled under the constraint $\mathcal{C}=\{u\in X: \|u\|\leq
M\}$. If the operator $K$ is weakly closed and injective, then the sequence
$\{u_{\eta^\ast(\delta )}^\delta\}_{\delta}$ of solutions converges weakly to
$u^\dagger$.
\end{thm}
\begin{proof}
Lemma \ref{lem:phiconv} implies $\|K(u_{\eta^\ast}^\delta)-g^\delta\| \rightarrow0$ as
$\delta\rightarrow0$. The a priori bound $\|u_{\eta^*}^\delta\|\leq M$ from the
constraint $\mathcal{C}$ implies the existence of a subsequence of
$\{u_{\eta^\ast}^\delta\}$, also denoted by $\{u_{\eta^\ast}^\delta\}$, and some $u^\ast$
such that $u_{\eta^\ast}^\delta\rightarrow u^\ast$ weakly. However, the weak closedness
of the operator $K$ and weak lower semi-continuity of norms, yield
$\|K(u^\ast)-g^\dagger\|=0$. Hence, $K(u^\ast)=g^\dagger$, which together with the
injectivity of the operator $K$ implies $u^\ast=u^\dagger$. Since every subsequence has a
subsequence converging weakly to $u^\dagger$, the whole sequence converges weakly to
$u^\dagger$. This concludes the proof of the theorem.
\end{proof}

Hence, the balancing principle is consistent provided that there exists a known upper
bound on the solution $u^\dagger$, which is often available from physical considerations.
This provides partial justification of its promising empirical results
\cite{ItoJinTakeuchi:2010}. In view of the uniform bound on the sequence
$\{\eta^*(\delta)\}$ in the defining relation \eqref{eqn:bal}, $\{\eta^*(\delta)\}$
naturally contains a convergent subsequence. However, it remains unclear whether the
(sub)sequence $\{\eta^\ast(\delta)\}$ will also tend to zero as the noise level $\delta$
vanishes.

\subsubsection{Hanke-Raus rule}

The Hanke-Raus rule \cite{HankeRaus:1996,EnglHankNeubauer:1996} is based on error
estimation: the squared residual $\|K(u_\eta^\delta)-g^\delta\|^2$ divided by the
regularization parameter $\eta$ behaves like an estimate for the total error (cf. Theorem
\ref{thm:apriori}). Hence, it chooses an optimal regularization parameter $\eta$ by
\begin{equation}\label{eqn:hankeraus}
  \eta^\ast=\arg\min_{\eta\in[0,\|K\|^2]}\frac{\|K(u_\eta^\delta)-g^\delta\|^2}{\eta}.
\end{equation}

We have the following a posteriori error estimate for the rule \eqref{eqn:hankeraus}.
\begin{thm}\label{thm:hr}
Let the conditions in Lemma \ref{lem:err} be fulfilled, $\eta^\ast$ be determined by rule
\eqref{eqn:hankeraus}, and $\delta_\ast=\|K(u_{\eta^\ast}^\delta) -g^\delta\|\neq0$ be
the realized residual. Then for any small noise level $\delta$, there holds
\begin{equation*}
   \|u_{\eta^\ast}^\delta-u^\dagger\|\leq c\left(\epsilon^{-\frac{1}{2}}\|w\|^\frac{1}{2}+{\epsilon'}^{-\frac{1}{2}}
   \tfrac{\delta}{\delta_\ast}\right)\max(\delta,\delta^\ast)^\frac{1}{2}.
\end{equation*}
\end{thm}
\begin{proof}
As in the proof of Theorem \ref{thm:bal}, we deduce that the error
$\|u_{\eta^\ast}-u^\dagger\|$ satisfies
\begin{equation*}
        \|u_{\eta^\ast}-u^\dagger\|\leq \tfrac{1}{\sqrt{\epsilon}}
        \left(2\tfrac{\sqrt{1+c_s}}{\sqrt{c_s}}\sqrt{\|w\|}\max(\delta,\delta_\ast)^\frac{1}{2}
    +\tfrac{2+2c_s}{c_s}\sqrt{c_r}\max(\delta,\delta_\ast)\right).
\end{equation*}

Next we estimate the error $\|u_{\eta^\ast}^\delta-u_{\eta^\ast}\|$. The definitions of
$\eta^\ast$ and $\delta_\ast$ indicate
\begin{equation*}
  \frac{\delta_\ast^2}{\eta^\ast}\leq\frac{\|K(u_{\tilde{\eta}}^\delta)-g^\delta\|^2}{\tilde{\eta}}
\end{equation*}
for any $\tilde{\eta}\in[0,\|K\|^2]$. By taking $\tilde{\eta}=\delta$ in the inequality
and noting Lemma \ref{lem:err}, we deduce
\begin{equation*}
  \begin{aligned}
    (\eta^\ast)^{-1}&\leq \delta_\ast^{-2}\delta^{-1}\|K(u_\delta^\delta)-g^\delta\|^2\\
      &\leq\delta_\ast^{-2}\delta^{-1}\left(\|K(u_\delta^\delta)-K(u_\delta)\|+\|K(u_\delta)-g^\dagger\|+\|g^\dagger-g^\delta\|\right)^2\\
      &\leq\delta_\ast^{-2}\delta^{-1}\left(\tfrac{2\delta}{c_s}+\tfrac{2\delta}{1-2c_r\delta}\|w\|+\delta\right)^2\\
      &=\left(\tfrac{2+c_s}{c_s}+\tfrac{2}{1-2c_r\delta}\|w\|\right)^2\delta\delta_\ast^{-2}.
  \end{aligned}
\end{equation*}
Using again Lemma \ref{lem:err}, we arrive at the following estimate
\begin{equation*}
  \begin{aligned}
       \|u_{\eta^\ast}^\delta-u_{\eta^\ast}\|&\leq\tfrac{1}{\sqrt{c_s\epsilon'}}\tfrac{1}{\sqrt{\eta^\ast}}\delta\\
       &\leq\tfrac{1}{\sqrt{c_s\epsilon'}}\tfrac{\delta}{\delta_\ast}\left(\tfrac{2+c_s}{
       c_s}+\tfrac{2}{1-2c_r\delta}\|w\|\right)\delta^\frac{1}{2}\\
       &\leq\tfrac{1}{\sqrt{c_s\epsilon'}}\tfrac{\delta}{\delta_\ast}\left(\tfrac{2+c_s}{
       c_s}+\tfrac{2}{1-2c_r\delta}\|w\|\right)\max(\delta,\delta_\ast)^\frac{1}{2}.
  \end{aligned}
\end{equation*}
After setting $c=\max(2\tfrac{\sqrt{1+c_s}}{\sqrt{c_s}}\sqrt{\|w\|} +\tfrac{2+2c_s}{c_s}
\sqrt{c_r}\max(\delta,\delta_\ast)^\frac{1}{2},\tfrac{1}{\sqrt{c_s}}
(\frac{2+c_s}{c_s}+\tfrac{2}{1-2c_r\delta}\|w\|))$, the desired assertion follows from
these two estimates and the triangle inequality.
\end{proof}

\section{A class of nonlinear parameter identification problems}\label{sec:class}

Now we revisit the source condition \eqref{scon} and nonlinearity condition \eqref{ncon}
for a general class of nonlinear parameter identification problems. The features of the
source and nonlinearity conditions are illuminated by utilizing the specific structure of
the adjoint operator $K'(u^\dagger)^\ast$. Then we specialize to problems with bilinear
structures, and show the unnecessity of the source representer $w$ for (numerically)
evaluating the nonlinearity term $\langle w,E(u,u^\dagger)\rangle$. Here we shall focus
on derivations in an abstract setting, and refer to Section \ref{sec:exam} for concrete
examples.

\subsection{A general class of parameter identification problems}
Generically, nonlinear parameter identification problems can be described by
\begin{equation*}
  \left\{\begin{array}{ll}
     e(u,y) = 0,\\
     K(u) = Cy(u),
  \end{array}\right.
\end{equation*}
where $e(u,y):X\times Y\rightarrow Y^*$ denotes a (differential) operator which is
differentiable with respect to both arguments $u$ and $y$, and the derivative $e_y$ is
assumed to be invertible. The notation $y(u)\in Y$ refers to the unique solution to the
operator equation $e(u,y)=0$ for a given $u$, and the operator $C$ is linear and bounded.
Typically, the operator $C$ represents an embedding or trace operator.

To make the source condition \eqref{scon} more precise and tangible, we compute the
derivative $K'(u)\delta u$ (with the help of the implicit function theorem) and the
adjoint operator $K'(u)^\ast$. Observe that the derivative $y'(u)\delta u$ of the
solution $y(u)$ with respect to $u$ in the direction $\delta u$ satisfies
\begin{equation*}
   e_u(u,y(u))\delta u + e_y(u,y(u))y'(u)\delta u=0,
\end{equation*}
from which follows the derivative formula
\begin{equation*}
  y'(u)\delta u = -(e_y(u,y(u)))^{-1}e_u(u,y(u))\delta u.
\end{equation*}
Consequently, we arrive at the following explicit representation
\begin{equation*}
   K'(u)\delta u = -C(e_y(u,y(u)))^{-1}e_u(u,y(u))\delta u.
\end{equation*}
Obviously, the adjoint operator $K'(u)^*$ is given by
\begin{equation*}
   K'(u)^\ast w = - e_u(u,y(u))^\ast(e_y(u,y(u)))^{-\ast}C^\ast w.
\end{equation*}

With the expression for the adjoint operator $K'(u)^\ast$, the source condition
\eqref{scon}, i.e., $K'(u^\dagger)^\ast w = u^\dagger-\mu^\dagger$, can be expressed more
explicitly as
\begin{equation*}
   - e_u(u^\dagger,y(u^\dagger))^\ast(e_y(u^\dagger,y(u^\dagger)))^{-\ast}C^\ast w = u^\dagger-\mu^\dagger.
\end{equation*}
This identity remains valid by setting $ \rho = -(e_y(u^\dagger,
y(u^\dagger)))^{-*}C^\ast w$. In other words, instead of the source condition
\eqref{scon}, we require the existence of $\rho\in Y$ and $\mu^\dagger$ such that
\begin{equation}\label{nscon}
   e_u(u^\dagger,y(u^\dagger))^\ast \rho = u^\dagger - \mu^\dagger,
\end{equation}
and $\langle \mu^\dagger,u-u^\dagger\rangle\geq0$ for any $u\in \mathcal{C}$. This
identity represents an alternative (new) source condition. A distinct feature of this
approach is that potentially less regularity is imposed on $\rho$, instead of on $w$.
This follows from the observation that the existence of $\rho\in Y$ does not necessarily
guarantee the existence of an element $w\in H$ satisfying $\rho=-(e_y(u^\dagger,
y(u^\dagger)))^{-*}C^\ast w$ due to possibly extra smoothing property of the operator
$(e_y(u^\dagger,y(u^\dagger)))^{-*}$, see Example \ref{exam:robin} for an illustration.
Conversely, the existence of $w$ always implies the existence of $\rho$ satisfying the
new source condition \eqref{nscon}. Therefore, it opens an avenue to relax the regularity
requirement of the source representer. Such a source condition underlies the main idea of
the interesting approach in \cite{EnglZou:2000} for a parabolic inverse problem.

Under the source condition \eqref{nscon}, we have
\begin{equation*}
   \begin{aligned}
      \langle w, E(u,u^\dagger)\rangle & = \langle w, Cy(u)-Cy(u^\dagger)-Cy'(u^\dagger)(u-u^\dagger)\rangle\\
          & = \langle C^\ast w, y(u)-y(u^\dagger)-y'(u^\dagger)(u-u^\dagger)\rangle\\
          & = \langle (e_{y}(u^\dagger,y(u^\dagger)))^{-\ast}C^\ast w,e_y(u^\dagger,y(u^\dagger))
          (y(u)-y(u^\dagger)-y'(u^\dagger)(u-u^\dagger))\rangle\\
          & = -\langle \rho,e_y(u^\dagger,y(u^\dagger))
          (y(u)-y(u^\dagger)-y'(u^\dagger)(u-u^\dagger))\rangle.
   \end{aligned}
\end{equation*}
Accordingly, the nonlinearity condition \eqref{ncon} can be expressed by
\begin{equation} \label{nncon}
  \begin{aligned}
     \tfrac{c_r}{2}\|K(u)-K(u^\dagger)\|^2
     +\langle \rho, e_y(u^\dagger,y(u^\dagger))&(y(u)-y(u^\dagger)-y'(u^\dagger)(u-u^\dagger))\rangle\\
     \quad\quad &+\tfrac{1}{2}\|u-u^\dagger\|^2 +\langle \mu^\dagger,u-u^\dagger\rangle
     \geq\tfrac{\epsilon}{2}\|u-u^\dagger\|^2\quad \forall u\in\mathcal{C}.
  \end{aligned}
\end{equation}
Therefore, the term $\langle \rho, e_y(u^\dagger,y(u^\dagger)) (y(u)-y(u^\dagger)-
y'(u^\dagger)(u-u^\dagger))\rangle$ will play an important role in studying the degree of
nonlinearity of the operator $K$, and in analyzing related Tikhonov regularization
methods. We shall illustrate its usage in Example \ref{exam:eit}. We would like to point
out that the nonlinearity condition \eqref{nncon} can be regarded as the (weak) limit of
the second-order sufficient condition \eqref{eqn:suffccond}, which in the current context
reads
\begin{equation*}
  \begin{aligned}
     \langle e_y(u_\eta,&y(u_\eta))^{-*}C^\ast(K(u_\eta)-g^\dagger),
     e_y(u_\eta,y(u_\eta))(y(u)-y(u_\eta)-y'(u_\eta)
     (u-u_\eta))\rangle\\
      &+\tfrac{1}{2}\|K(u_\eta)-K(u)\|^2+\tfrac{\eta}{2}\|u_\eta-u\|^2 +\eta\langle\mu_\eta,u-u_\eta\rangle
     \geq\tfrac{c_s}{2}\|K(u_\eta)-K(u)\|^2+\tfrac{\epsilon'\eta}{2}\|u-u_\eta\|^2.
  \end{aligned}
\end{equation*}

The source condition \eqref{nscon} together with the nonlinearity condition \eqref{nncon}
can yield identical convergence rates for nonlinear Tikhonov models as conditions
\eqref{scon} and \eqref{ncon}, since conditions \eqref{nscon} and \eqref{nncon} are
exactly the representations of conditions \eqref{scon} and \eqref{ncon} in the context of
parameter identifications. The main changes to the proofs are the following two key
identities
\begin{equation*}
   \begin{aligned}
      \langle w,K'(u^\dagger)(u-u^\dagger)\rangle&=\langle\rho, e_u(u^\dagger,y(u^\dagger))
      (u-u^\dagger)\rangle(=\langle u^\dagger-\mu^\dagger,u-u^\dagger\rangle),\\
      \langle w,K(u)-K(u^\dagger)\rangle & = - \langle \rho, e_y(u^\dagger,y(u^\dagger))(y(u)-y(u^\dagger))\rangle,
   \end{aligned}
\end{equation*}
and the remaining steps proceed identically.

In the rest, we further specialize to the case where the operator equation $e(u,y)=0$
assumes the form
\begin{equation*}
   A(u)y - f=0.
\end{equation*}
A lot of parameter identification problems for linear partial differential equations
(systems) can be cast into this abstract model, e.g., the second-order elliptic operator
$A(u)y=-\nabla\cdot(a(x)\nabla y)+\mathbf{b}(x)\cdot\nabla y + c(x)y$ with the parameter
$u$ being one or some combinations of $a(x),\,\mathbf{b}(x)$ and $c(x)$. Then upon
denoting the derivative of $A(u)$ with respect to $u$ by $A'(u)$, we have
\begin{equation*}
   e_u(u,y(u))\delta u = A'(u)\delta u y(u)
\end{equation*}
and
\begin{equation*}
   e_y(u,y(u)) = A(u).
\end{equation*}
The derivative $A'(u)\delta u y(u)$ can be either local (separable) or nonlocal. For
example, in the former category, $A(u)y=(-\Delta + u)y$ with $A'(u)\delta u
y(u)=y(u)\delta u$. The case $A(u)y=-\nabla\cdot(u\nabla y)$ with $A'(u)\delta u
y(u)=-\nabla\cdot(\delta u\nabla y(u))$ belongs to the latter category. The local case
will be further discussed in Section \ref{subsec:bilin}. Consequently, the (new) source
and nonlinearity conditions respectively simplify to
\begin{equation*}
   e_u(u^\dagger,y(u^\dagger))^\ast \rho = u^\dagger - \mu^\dagger
\end{equation*}
and
\begin{equation*}
  \begin{aligned}
     \tfrac{c_r}{2}\|K(u)-K(u^\dagger)\|^2
     +\langle \rho, A(u^\dagger)&(y(u)-y(u^\dagger)-y'(u^\dagger)(u-u^\dagger))\rangle\\
     \quad\quad &+\tfrac{1}{2}\|u-u^\dagger\|^2 +\langle \mu^\dagger,u-u^\dagger\rangle
     \geq\tfrac{\epsilon}{2}\|u-u^\dagger\|^2\quad \forall u\in\mathcal{C}.
  \end{aligned}
\end{equation*}

\subsection{Bilinear problems}\label{subsec:bilin}

Here we elaborate the structure of the crucial nonlinearity term $\langle
w,E(u,u^\dagger)\rangle$ in the proposed nonlinearity condition \eqref{ncon}.
Interestingly, it admits a representation without resorting to the source representer $w$
for bilinear problems. Specifically, the following class of inverse problems is
considered. Let the operator $e(u,y)$ be (affine) bilinear with respect to the arguments
$u$ and $y$ for fixed $y$ and $u$, respectively, and for a given $u$, $e_u(u,y)$ is
defined pointwise (local/separable).

We begin with the second-order error $E(u,u^\dagger)$ for bilinear problems. The bilinear
structure of the operator $e(u,y)$ implies
\begin{equation*}
  \begin{aligned}
     0 & = e(u,y(u)) - e(u^\dagger,y(u^\dagger))\\
       & = e_y(u^\dagger,y(u^\dagger))(y(u)-y(u^\dagger))+ e_u(u,y(u))(u-u^\dagger),
  \end{aligned}
\end{equation*}
i.e.,
\begin{equation*}
y(u)-y(u^\dagger) = - (e_y(u^\dagger,y(u^\dagger)))^{-1}e_u(u,y(u))(u-u^\dagger).
\end{equation*}
Therefore, we deduce that
\begin{equation*}
  \begin{aligned}
     E(u,u^\dagger) & = K(u)-K(u^\dagger) - K'(u^\dagger)(u-u^\dagger)\\
       & = Cy(u) - Cy(u^\dagger) + C(e_y(u^\dagger,y(u^\dagger)))^{-1}
             e_u(u^\dagger,y(u^\dagger))(u-u^\dagger),\\
       & = -C(e_y(u^\dagger,y(u^\dagger)))^{-1}e_u(u,y(u^\dagger))(u-u^\dagger)+
            C(e_y(u^\dagger,y(u^\dagger)))^{-1}e_u(u^\dagger,y(u^\dagger))(u-u^\dagger)\\
       & = -C(e_y(u^\dagger,y(u^\dagger)))^{-1}(e_u(u,y(u))
           -e_u(u^\dagger,y(u^\dagger)))(u-u^\dagger).
  \end{aligned}
\end{equation*}

With the help of the preceding three relations, the source condition $K'(u^\dagger)^\ast
w = u^\dagger-\mu^\dagger$ and locality (separability) of $e_u(u,y(u))$, we get
\begin{equation*}
  \begin{aligned}
   \langle w,E(u,u^\dagger)\rangle &= \langle w, -C(e_y(u^\dagger,y(u^\dagger)))^{-1}
         (e_u(u,y(u))-e_u(u^\dagger,y(u^\dagger)))(u-u^\dagger)\rangle\\
    & = \left\langle -e_u(u^\dagger,y(u^\dagger))^\ast(e_y(u^\dagger,y(u^\dagger)))^{-\ast}
         C^\ast w,\frac{e_u(u,y(u))-e_u(u^\dagger,y(u^\dagger))}{
         e_u(u^\dagger,y(u^\dagger))}(u-u^\dagger)\right\rangle\\
    & = \left\langle K'(u^\dagger)^\ast w,
        \frac{e_u(u,y(u))-e_u(u^\dagger,y(u^\dagger))
        }{e_u(u^\dagger,y(u^\dagger))}(u-u^\dagger)\right\rangle\\
    & = \left\langle u^\dagger - \mu^\dagger,
        \frac{e_u(u,y(u))-e_u(u^\dagger,y(u^\dagger))}{
        e_u(u^\dagger,y(u^\dagger))}(u-u^\dagger)\right\rangle.
  \end{aligned}
\end{equation*}
Therefore, we have arrived at the following concise representation
\begin{equation}\label{eqn:wrepres}
   \langle w,E(u,u^\dagger)\rangle = \left\langle u^\dagger - \mu^\dagger,
    \frac{e_u(u,y(u))-e_u(u^\dagger,y(u^\dagger))
    }{e_u(u^\dagger,y(u^\dagger))}(u-u^\dagger)\right\rangle.
\end{equation}

A remarkable observation of the derivations is that the source representer $w$ actually
is not needed for evaluating $\langle w,E(u,u^\dagger)\rangle$, which enables possible
numerical verification of the nonlinearity condition \eqref{ncon}. Note that even if we
do know the exact solution $u^\dagger$, the representer $w$ is still not directly
accessible since the operator equation $K'(u^\dagger)^*w =u^\dagger$ is generally also
ill-posed. Hence, the representation \eqref{eqn:wrepres} is of much practical
significance. Another important consequence is that it may enable estimates of type
\eqref{cond:clsnon2}, thereby validating the nonlinearity condition \eqref{ncon}. This
can be achieved by applying H\"{o}lder-type inequality if the image of $K(u)$ and the
coefficient $u$ share the domain of definition, e.g., in recovering the potential/leading
coefficient in an elliptic equation from distributed measurements in the domain, see
Example \ref{exam:eit} for an illustration.

Finally, we point out that formally the representation \eqref{eqn:wrepres} can be
regarded as the limit of
\begin{equation*}
  \left\langle u_\eta-\mu_\eta,\frac{e_u(u,y(u))-
    e_u(u_\eta,y(u_\eta))}{e_u(u_\eta,y(u_\eta))}(u-u_\eta)\right\rangle
\end{equation*}
as $\eta$ goes to zero, which might be computationally amenable, and hence enable
possible numerical verification of the second-order sufficient condition
\eqref{eqn:suffccond}.

\section{Examples}\label{sec:exam}

In this section, we illuminate the nonlinearity condition \eqref{ncon} and the structure
of the term $\langle w,E(u,u^\dagger)\rangle$ with examples, and discuss the usage of the
source and nonlinearity conditions \eqref{nscon} and \eqref{nncon}.

First, we give a one-dimensional example where the smallness assumption ($L\|w\|<1$) in
the classical condition \eqref{cond:clsnon} is violated while the proposed nonlinearity
condition \eqref{ncon} is always true.

\begin{exam}
Let the nonlinear operator $K:\mathbb{R}\rightarrow\mathbb{R}$ be given by
\begin{equation*}
 K(u)=\epsilon u(1-u),
\end{equation*}
where $\epsilon>0$. The solution depends sensitively on $u$ if $\epsilon$ is very small,
hence it mimics the ill-posed behavior of inverse problems. Let the exact data
$g^\dagger$ $($necessarily smaller than $\frac{\epsilon}{4}$$)$ be given, then the
minimum-norm solution $u^\dagger$ is given by
\begin{equation*}
u^\dagger = \tfrac{1}{2}\left(1-\sqrt{1-4\epsilon^{-1}g^\dagger}\right).
\end{equation*}
It is easy to verify that
\begin{equation*}
K'(u^\dagger) = \epsilon(1-2u^\dagger),
\end{equation*}
and
\begin{equation*}
\begin{aligned}
 E(u,u^\dagger)&=K(u)-K(u^\dagger)-K'(u^\dagger)(u-u^\dagger)\\
 &=\epsilon u(1-u)-\epsilon u^\dagger(1-u^\dagger)-\epsilon(1-2u^\dagger)(u-u^\dagger)\\
 &=-\epsilon(u-u^\dagger)^2.
\end{aligned}
\end{equation*}
Now the source condition $K'(u^\dagger)^\ast w = u^\dagger$ implies that the source
representer $w$ is given by $w = \frac{u^\dagger}{ \epsilon (1-2u^\dagger)}$. Therefore,
the nonlinearity term $\langle w,E(u,u^\dagger)\rangle$ is given by
\begin{equation*}
  \langle w,E(u,u^\dagger)\rangle=\frac{-u^\dagger}{1-2u^\dagger}(u-u^\dagger)^2
\end{equation*}
which is smaller than zero for a fixed but sufficiently small $g^\dagger>0$. Moreover,
the prefactor $\left|\frac{u^\dagger}{1-2u^\dagger}\right|$ can be made arbitrarily
large, thereby indicating that the smallness assumption $(L\|w\|<1)$ can never be
satisfied then. Actually, the Lipschitz constant $L$ of $K'(u)$ is $L=2\epsilon$, and
$L|w|=\frac{2u^\dagger}{|1-2u^\dagger|}$, which can be arbitrarily large $($if
$u^\dagger$ is sufficiently close to $\frac{1}{2}$$)$, and thus the classical
nonlinearity condition \eqref{cond:clsnon} is violated. This shows that the proposed
nonlinearity condition \eqref{ncon} is indeed much weaker than the classical one.

A direct calculation shows in the second-order sufficient condition
\eqref{eqn:suffccond},
\begin{equation*}
  \langle K(u_\eta)-g^\dagger,E(u,u_\eta)\rangle = \eta\frac{u_\eta}{1-2u_\eta}(u-u_\eta)^2.
\end{equation*}
Observe that the form of $\frac{1}{\eta}\langle K(u_\eta)-g^\dagger,E(u,u_\eta)\rangle$
coincides with that of $-\langle w, E(u,u^\dagger)\rangle$. With this explicit
representation at hand, the nonnegativity of the term $\langle K(u_\eta)-
g^\dagger,E(u,u_\eta)\rangle$, and thus the second-order sufficient condition
\eqref{eqn:suffccond}, can be numerically verified for a given $g^\dagger$ and every
possible $\eta$ since the Tikhonov minimizer $u_\eta$ can be found by solving a cubic
equation.
\end{exam}

Next we consider an elliptic parameter identification problem to show that the smallness
assumption $L\|w\|<1$ of the classical nonlinearity condition \eqref{cond:clsnon} is
unnecessary by deriving an explicit representation of the nonlinearity term $\langle w,
E(u,u^\dagger)\rangle$. The derivations also illustrate clearly structural properties
developed in the abstract framework in Section \ref{subsec:bilin}.
\begin{exam}[Robin inverse problem]\label{exam:robin}
Let $\Omega\subset\mathbb{R}^2$ be an open bounded domain with a smooth boundary
$\Gamma$, which consists of two disjoint parts $\Gamma_i$ and $\Gamma_c$. We consider the
following elliptic equation
\begin{equation}\label{bdry}
  \begin{aligned}
   \left\{\begin{array}{ll}
       -\Delta y=0 & \mbox{ in } \Omega,\\
       \frac{\partial y}{\partial n}=f & \mbox{ on }\Gamma_c,\\
       \frac{\partial y}{\partial n}+u y=h& \mbox{ on }\Gamma_i.
   \end{array}\right.
  \end{aligned}
\end{equation}
We measure $g=y$ on $\Gamma_c$ and are interested in recovering the Robin coefficient
$u\in\mathcal{C}=\{u:u\geq c \}$ for some $c>0$ by means of Tikhonov regularization
\begin{equation*}
  \min_{u\in\mathcal{C}} \int_{\Gamma_c}|y-g^\delta|^2ds+\eta \int_{\Gamma_i} |u|^2ds.
\end{equation*}
It arises in corrosion detection and analysis of quenching process
\cite{Inglese:1997,JinZou:2010b}. Let $y(u)\in H^1(\Omega)$ be the solution to
\eqref{bdry}, and $\gamma_{\Gamma_c}$ be the trace operator to the boundary $\Gamma_c$,
similarly $\gamma_{\Gamma_i}$. Then direct computation shows $K(u)=\gamma_{\Gamma_c}
y(u)$. It is easy to show that the forward operator $K:L^2(\Gamma_i)\mapsto
L^2(\Gamma_c)$ is Fr\'{e}chet differentiable and the derivative is Lipschitz continuous.
Moreover, straightforward computations give
\begin{equation*}
  \begin{aligned}
     K'(u^\dagger)\delta u & =\gamma_{\Gamma_c}\tilde{z}(u^\dagger),\\
     E(u,u^\dagger)&=\gamma_{\Gamma_c}v(u,u^\dagger),\\
     K^\prime(u^\dagger)^*w&=-\gamma_{\Gamma_i}(y(u^\dagger)z(u^\dagger)),
  \end{aligned}
\end{equation*}
where the functions $\tilde{z}(u^\dagger)$, $v(u,u^\dagger)$ and $z(u^\dagger)\in
H^1(\Omega)$ satisfy
\begin{equation*}
  \begin{aligned}
  \int_\Omega \nabla \tilde{z}\cdot\nabla \tilde{v}dx +\int_{\Gamma_i}u^\dagger\tilde{z} \tilde{v}ds &= -\int_{\Gamma_i}
     \delta u y(u)\tilde{v}ds\quad\forall \tilde{v}\in H^1(\Omega),\\
  \int_\Omega \nabla v\cdot\nabla \tilde{v}dx +\int_{\Gamma_i}u^\dagger v \tilde{v}ds &= -\int_{\Gamma_i}
     (u-u^\dagger)(y(u)-y(u^\dagger))\tilde{v}ds\quad\forall \tilde{v}\in H^1(\Omega),\\
  \int_\Omega \nabla z\cdot\nabla \tilde{v}dx +\int_{\Gamma_i}u^\dagger z \tilde{v}ds &= \int_{\Gamma_c}
     w\tilde{v}ds\quad \forall\tilde{v}\in H^1(\Omega).
  \end{aligned}
\end{equation*}
Assume that the source condition \eqref{ncon} holds with the representer $w\in
L^2(\Gamma_c)$. Then by setting $\tilde{v}=z(u^\dagger)$ and $\tilde{v}=v(u,u^\dagger)$
respectively in the their weak formulations, it follows that
\begin{equation*}
   \langle w,E(u,u^\dagger)\rangle_{L^2(\Gamma_c)} = \left\langle u^\dagger-\mu^\dagger,
   (u-u^\dagger)\frac{y(u)-y(u^\dagger)}{y(u^\dagger)}\right\rangle_{L^2(\Gamma_i)}.
\end{equation*}
Hence, the term $\langle w,E(u,u^\dagger)\rangle$ exhibits the desired structure, cf.
\eqref{eqn:wrepres}. Next, by maximum principle, the solution $y(u)$ to \eqref{bdry} is
positive for positive $f$ and $h$. Moreover there holds
\begin{equation*}
   \int_\Omega |\nabla (y(u)-y(u^\dagger))|^2dx+\int_{\Gamma_i}u |y(u)-y(u^\dagger)|^2ds
   +\int_{\Gamma_i}(u-u^\dagger)y(u^\dagger)(y(u)-y(u^\dagger))ds=0.
\end{equation*}
Note also the monotonicity relation, i.e., if $u\ge u^\dagger$, then $y(u)\leq
y(u^\dagger)$. It follows from the above two relations that if
$u^\dagger-\mu^\dagger\ge0$, then
\begin{equation*}
    -\langle w,E(u,u^\dagger)\rangle\ge 0.
\end{equation*}
This shows that the nonlinearity condition \eqref{ncon} holds without resorting to the
smallness condition $L\|w\|<1$ in the classical nonlinearity condition
\eqref{cond:clsnon} under the designated circumstance.

Next we contrast the source condition \eqref{nscon} with the conventional one
\eqref{scon}. The operator $e(u,y)$ is bilinear, and $e_u(u,y(u))=\gamma_{\Gamma_i}y(u)$,
$e_u(u,y(u))^\ast\rho=\gamma_{\Gamma_i}(\rho y(u))$. Hence the new source condition
\eqref{nscon} requires the existence of some element $\rho\in H^1(\Omega)$ such that
\begin{equation*}
   \gamma_{\Gamma_i}(\rho y(u^\dagger)) = u^\dagger-\mu^\dagger.
\end{equation*}
This admits an easy interpretation: for $\gamma_{\Gamma_i}\rho$ to be fully determined,
$\gamma_{\Gamma_i}y(u^\dagger)$ cannot vanish, which is exactly the identifiability
condition (via Newton's law for convective heat transfer) \cite{JinZou:2010b}. The
representers $\rho$ and $w$ are related by $($in weak form$)$
\begin{equation*}
  \int_\Omega\nabla \rho\cdot\nabla \tilde{v}dx + \int_{\Gamma_i}u^\dagger\rho\tilde{v}ds = -\int_{\Gamma_c}w\tilde{v}ds
  \quad \forall \tilde{v}\in H^1(\Omega).
\end{equation*}
This relation shows clearly the different regularity assumptions on $\rho$ and $w$: the
existence of $\rho\in H^1(\Omega)$ does not actually guarantee the existence of $w\in
L^2(\Gamma_c)$. To ensure the existence of $w\in L^2(\Gamma_c)$, one necessarily needs
higher regularity on $\rho$ than $H^1(\Omega)$, presumably $\rho\in
H^\frac{3}{2}(\Omega)$. Conversely, the existence of $w\in L^2(\Gamma_c)$ automatically
ensures the existence of $\rho\in H^1(\Omega)$.
\end{exam}

Next we give an example of inverse medium scattering to show the same structure of the
nonlinearity term $\langle w,E(u,u^\dagger)\rangle$ but with less definitiveness.
\begin{exam}[Inverse scattering problem] Here we consider the two-dimensional time-harmonic inverse
scattering problem of determining the index of refraction $n^2$ from near-field scattered
field data, given one incident field $y^i$ \cite{ColtonKress:1998}. Let $y=y(x)$ denote
the transverse mode wave and satisfy
\begin{equation*}
   \Delta y + n^2k^2 y = 0.
\end{equation*}
Let the incident plane wave be $y^{i}=e^{kx\cdot d}$ with $d=(d_1,d_2)\in\mathbb{S}^1$
being the incident direction. Then for the complex coefficient $u=(n^2-1)k^2$ with its
support within $\Omega\subset\mathbb{R}^2$, the total field $y=y^{tot}$ satisfies
\begin{equation*}
   y= y^i + \int_\Omega G(x,z)u(z)y(z)dz,
\end{equation*}
where $G(x,z)$ is the free space fundamental solution, i.e., $G(x,z)=\frac{i}{4}
H_0^1(k|x-z|)$, the Hankel function of the first kind and zeroth order. The inverse
problem is to determine the refraction coefficient $u$ from the scattered field
\begin{equation*}
  y^s(x) = \int_\Omega G(x,z)u(z)y(z)dz
\end{equation*}
measured on a near-field boundary $\Gamma$. Consequently, we have $K(u)=\gamma_\Gamma
y^s(x)\in L^2(\Gamma)$. The Tikhonov approach for recovering $u$ takes the form
\begin{equation*}
   \min_{u\in \mathcal{C}} \int_{\Gamma}|K(u)-g^\delta|^2ds +\eta\int_\Omega|u|^2dx.
\end{equation*}
Here $g^\delta$ denotes the measured scattered field, and the constraint set
$\mathcal{C}$ is taken to be $\mathcal{C}=\{u\in L^\infty(\Omega): \Re(u)\geq 0,
\mathrm{supp}(u)\subset\subset\Omega\}$. It can be shown that the forward operator
$K:L^2(\Omega)\mapsto L^2(\Gamma)$ is Fr\'{e}chet differentiable, and the derivative is
Lipschitz continuous on $\mathcal{C}$. Now let $G^\dagger(x,z)$ be the fundamental
solution to the elliptic operator $\Delta + k^2 + u^\dagger$. Then we can deduce
\begin{equation*}
   \begin{aligned}
      K'(u^\dagger)\delta u  &= -\int_\Omega G^\dagger(x,z)\delta u y(u^\dagger)dz\quad x\in\Gamma,\\
      E(u,u^\dagger) & = -\int_\Omega G^\dagger(x,z)(u-u^\dagger)(y(u)-y(u^\dagger))dz\quad x\in\Gamma,\\
      K'(u^\dagger)^\ast w & = -\overline{y(u^\dagger)}\int_\Gamma \overline{G^\dagger(x,z)}w(x)dx,
   \end{aligned}
\end{equation*}
where $\bar{\quad}$ refers to taking complex conjugate. Then, by the source condition
$K'(u^\dagger)^\ast w=u^\dagger-\mu^\dagger$, we get
\begin{equation*}
   \begin{aligned}
         \langle w, E(u,u^\dagger)\rangle_{L^2(\Gamma)}
              &=-\int_\Gamma w(x) \overline{\int_\Omega G^\dagger(x,z)(u(z)-u^\dagger(z))(y(u)(z)-y(u^\dagger)(z))dz}dx\\
              &=-\int_\Omega \overline{y(u^\dagger)}\int_\Gamma\overline{G^\dagger(x,z)}w(x)dx(\overline{u(z)}-\overline{u^\dagger(z)})
              \frac{\overline{y(u)(z)}-\overline{y(u^\dagger)(z)}}{\overline{y(u^\dagger)(z)}}dz\\
              &=\left\langle u^\dagger-\mu^\dagger,(u-u^\dagger)\frac{y(u)-
               y(u^\dagger)}{y(u^\dagger)}\right\rangle_{L^2(\Omega)}.
   \end{aligned}
\end{equation*}
Note that the structure of $\langle w,E(u,u^\dagger) \rangle_{L^2(\Gamma)}$ coincides
with that in Example \ref{exam:robin}, which further corroborates the theory for bilinear
problems in Section \ref{subsec:bilin}. However, an analogous argument for definitive
sign is missing since the maximum principle does not hold for the Helmholtz equation.
Nonetheless, one might still expect some norm estimate of the form \eqref{cond:clsnon2},
which remains open. In particular, then a small $u^\dagger-\mu^\dagger$ would imply the
nonlinearity condition \eqref{ncon}.
\end{exam}

The last example shows the use of the source condition \eqref{nscon} and nonlinearity
condition \eqref{nncon}.
\begin{exam}[Inverse conductivity problem]\label{exam:eit}
Let $\Omega\subset\mathbb{R}^2$ be an open bounded domain with a smooth boundary
$\Gamma$. We consider the following elliptic equation
\begin{equation*}
  \begin{aligned}
   \left\{\begin{array}{ll}
       -\nabla \cdot(u\nabla  y)=f & \mbox{ in } \Omega,\\
       y = 0 & \mbox{ on }\Gamma.
   \end{array}\right.
  \end{aligned}
\end{equation*}
Let $y(u)\in H_0^1(\Omega)$ be the solution. We measure $y$ $($denoted by $g^\delta\in
H_0^1(\Omega)$$)$ in the domain $\Omega$ with
$\|\nabla(g^\delta-y(u^\dagger))\|_{L^2(\Omega)}\leq\delta$, i.e., $K(u)=y(u)$, and are
interested in recovering the conductivity $u\in\mathcal{C}=\{u\in H^1(\Omega): c_0\leq
u\leq c_1 \}$ for some finite $c_0,c_1>0$ by means of Tikhonov regularization
\begin{equation*}
  \min_{u\in\mathcal{C}} \int_{\Omega}|\nabla(K(u)-g^\delta)|^2dx+\eta \int_\Omega|u|^2+|\nabla u|^2dx.
\end{equation*}
It arises in estimating permeability of underground flow and thermal conductivity in heat
transfer \cite{Yeh:1986}. It follows from Meyers' theorem \cite{Meyers:63} that the
operator $K:H^1(\Omega)\mapsto H_0^1(\Omega)$ is Fr\'{e}chet differentiable and the
derivative is Lipschitz continuous. The operator $A(u)$ is given by
$-\nabla\cdot(u\nabla\cdot)$. It is easy to see that
\begin{equation*}
   \begin{aligned}
      A'(u)\delta u y(u) &= -\nabla \cdot(\delta u\nabla y(u)),\\
      e_u(u,y(u))^\ast\rho  & = \nabla y(u)\cdot\nabla\rho.
   \end{aligned}
\end{equation*}
Consequently, the source condition \eqref{nscon} reads: there exists some $\rho\in
H_0^1(\Omega)$ such that
\begin{equation*}
  \nabla y(u^\dagger)\cdot\nabla \rho = (I-\Delta)u^\dagger-\mu^\dagger,
\end{equation*}
which amounts to the solvability condition $\nabla y(u^\dagger)\neq 0$ $($cf., e.g.,
\cite{Richter:1981,ItoKunisch:1994}$)$. The nonlinearity condition \eqref{nncon} is given
by
\begin{equation*}
  \begin{aligned}
     \tfrac{c_r}{2}\|\nabla(y(u)-y(u^\dagger))\|_{L^2(\Omega)}^2
     &-\langle u^\dagger\nabla \rho, \nabla E(u,u^\dagger)\rangle\\
     &+\tfrac{1}{2}\|u-u^\dagger\|_{H^1(\Omega)}^2 +\langle \mu^\dagger,u-u^\dagger\rangle
     \geq\tfrac{\epsilon}{2}\|u-u^\dagger\|_{H^1(\Omega)}^2\quad \forall u\in\mathcal{C},
  \end{aligned}
\end{equation*}
where $E(u,u^\dagger)=K(u)-K(u^\dagger)-K'(u^\dagger) (u-u^\dagger)$ is the second-order
error. By setting $\tilde{v}=\rho$ in the weak formulation of $E(u,u^\dagger)$, i.e.,
\begin{equation*}
  \int u^\dagger\nabla E(u,u^\dagger)\cdot\nabla \tilde{v} dx
  = -\int_\Omega(u-u^\dagger)\nabla(y(u)-y(u^\dagger))\cdot \nabla \tilde{v}dx\quad \forall \tilde{v}\in H_0^1(\Omega).
\end{equation*}
and applying the generalized H\"{o}lder's inequality and Sobolev embedding theorem, we
get
\begin{equation*}
   \begin{aligned}
      |\langle u^\dagger\nabla \rho,\nabla E(u,u^\dagger)\rangle| &\leq
      \|\nabla(y(u)-y(u^\dagger))\|_{L^2(\Omega)}\|u-u^\dagger\|_{L^q(\Omega)}\|\nabla \rho\|_{L^p(\Omega)}\\
      &\leq C\|\nabla\rho\|_{L^p(\Omega)}\|\nabla (y(u)-(y^\dagger))\|_{L^2(\Omega)}\|u-u^\dagger\|_{H^1(\Omega)},
   \end{aligned}
\end{equation*}
where the exponents $p,q>2$ satisfy $\frac{1}{p}+\frac{1}{q}=\frac{1}{2}$ $($the exponent
$p$ can be any number greater than $2$$)$. Therefore, we have established condition
\eqref{cond:clsnon2} for the inverse conductivity problem, and the nonlinearity condition
\eqref{nncon} holds provided that the source representer $\rho\in W_0^{1,p}(\Omega)$ for
some $p>2$. We especially note that the smallness of the representer $\rho$ is not
required for the nonlinearity condition \eqref{nncon} for this example. The convergence
theory in Section 3 implies a convergence rate $\|u_\eta^\delta-u^\dagger\|_{H^1(\Omega)}
\leq C\sqrt{\delta}$ for the Tikhonov model with the a priori choice $\eta\sim\delta$ and
the discrepancy principle.

Note that the classical source condition \eqref{scon} reads: there exists some $w\in
H_0^1(\Omega)$ such that
\begin{equation*}
  K'(u^\dagger)^*w = (I-\Delta) u^\dagger - \mu^\dagger,
\end{equation*}
or equivalently in the weak formulation
\begin{equation*}
  \langle \nabla K'(u^\dagger) h, \nabla w\rangle = \langle u^\dagger, h\rangle_{H^1(\Omega)}
  - \langle\mu^\dagger,h\rangle\quad \forall h\in H^1(\Omega).
\end{equation*}
This source condition is difficult to interpret due to the lack of an explicit
characterization of the range of the adjoint operator $K'(u^\dagger)^*$, as often is the
case of parameter identifications \cite{EnglKunischNeubauer:1989,EnglHankNeubauer:1996}.
Also the weak formulation of $K'(u^\dagger)h\in H_0^1(\Omega)$, i.e.,
\begin{equation*}
   \langle u^\dagger\nabla K'(u^\dagger)h,\nabla v\rangle = \langle h\nabla y(u^\dagger),\nabla v\rangle
   \quad \forall v\in H_0^1(\Omega),
\end{equation*}
does not directly help due to subtle differences in the relevant bilinear forms.
Nonetheless, the representers $\rho$ and $w$ are closely related by
\begin{equation*}
  \rho = (A(u^\dagger))^{-1}(-\Delta) w.
\end{equation*}
This relation indicates that the operator $(A(u^\dagger))^{-1}(-\Delta)$ renormalizes the
standard inner product $\langle\nabla\cdot,\nabla\cdot\rangle$ on $H_0^1(\Omega)$ to a
problem-adapted weighted inner product $\langle u^\dagger\nabla\cdot,
\nabla\cdot\rangle$, and thus facilitates the interpretation of the resulting source
condition. This shows clearly the advantage of the source condition \eqref{nscon}.
\end{exam}

\section{Concluding remarks}
In the paper, we have presented a new approach to constrained nonlinear Tikhonov
regularization on the basis of a second-order sufficient condition, which was suggested
as an alternative nonlinearity condition. The proposed approach allows deriving
convergence rates for several a priori and a posteriori parameter choice rules, including
discrepancy principle, balancing principle and Hanke-Raus rule, and thus it is useful in
analyzing Tikhonov models. The structures of the source condition and nonlinearity
condition were discussed for a general class of nonlinear parameter identification
problems, especially more transparent source and nonlinearity conditions were derived. It
was found that for bilinear problems, the source representer $w$ in the crucial
nonlinearity term $\langle w,E(u,u^\dagger)\rangle$ actually does not appear. The theory
was illustrated in detail on several concrete examples, including three exemplary
parameter identification problems for elliptic differential equations. It was shown that
the proposed nonlinearity condition can be much weaker than the classical one, and the
crucial term $\langle w,E(u,u^\dagger)\rangle$ can admit nice structures that are useful
for deriving error estimates.

\section*{Acknowledgements}
The authors are grateful to two anonymous referees whose constructive comments have led
to an improved presentation. The work of Bangti Jin is supported by Award No.
KUS-C1-016-04, made by King Abdullah University of Science and Technology (KAUST). A part
of the work was carried out during his visit at Graduate School of Mathematical Sciences,
The University of Tokyo, and he would like to thank Professor Masahiro Yamamoto for the
kind invitation and hospitality.

\bibliographystyle{abbrv}
\bibliography{tikh}

\end{document}